\documentclass[11pt, leqno, twoside]{article}
\usepackage[left=3cm,top=3cm,right=3cm,bottom=4.5cm]{geometry}
\usepackage{amssymb, amsmath} 
\usepackage{enumerate}
\usepackage{color}
\usepackage{amsthm}
\usepackage{ esint }

\newtheorem{theorem}{Theorem}[section]

\newtheorem{lem}[theorem]{Lemma}
\newtheorem{proposition}[theorem]{Proposition}
\newtheorem{prop}[theorem]{Proposition}

\newtheorem{defn}[theorem]{Definition}
\newtheorem{definition}[theorem]{Definition}

\numberwithin{equation}{section}

\newcommand{\N}{\mathbb{N}}
\newcommand{\R}{\mathbb{R}}
\newcommand{\Rn}{\R^n}
\newcommand{\Rnn}{\mathbb{R}^{n \times n}}

\newcommand{\mS}{\mathbb{S}}

\newcommand{\A}{\mathcal{A}}
\newcommand{\E}{\mathcal{E}}
\newcommand{\B}{\mathcal{B}}
\newcommand{\Ha}{\mathcal{H}}

\newcommand{\vn}{\vec{n}}
\newcommand{\vm}{\vec{m}}
\newcommand{\vl}{\vec{\ell}}
\newcommand{\hvpe}{\hat{\vp}_\eta}
\newcommand{\Ap}{\A_p}
\newcommand{\ssubset}{\subset \! \subset}
\newcommand{\weakcs}{\overset{*}{\rightharpoonup}}

\newcommand{\ve}{\varepsilon}
\newcommand{\vp}{\varphi}
\newcommand{\Om}{\Omega}

\DeclareMathOperator{\id}{id}
\DeclareMathOperator{\cof}{cof}
\DeclareMathOperator{\adj}{adj}
\DeclareMathOperator{\Div}{div}

\DeclareMathOperator{\imT}{im_{T}}
\DeclareMathOperator{\imG}{im_{G}}

\DeclareMathOperator{\mec}{mec}
\DeclareMathOperator{\nem}{nem}

\DeclareMathOperator{\curl}{curl}

\DeclareMathOperator{\tr}{tr}

\renewcommand{\leq}{\leqslant}
\renewcommand{\le}{\leqslant}
\renewcommand{\geq}{\geqslant}
\renewcommand{\ge}{\geqslant}

\newcommand{\dd}{\mathrm{d}}

\renewcommand{\epsilon}{\varepsilon }

\newcommand{\weakc}{\rightharpoonup}
\newcommand{\weakly}{\rightharpoonup}

\begin{document}

\title{Relaxation of nonlinear elastic energies involving deformed configuration and applications to nematic elastomers}

\author{Carlos Mora-Corral and Marcos Oliva
\\
\footnotesize Department of Mathematics, Faculty of Sciences, Universidad Aut\'onoma de Madrid, 28049 Madrid, Spain
}

\date{\today}
\maketitle

\begin{abstract}
We start from a variational model for nematic elastomers that involves two energies: mechanical and nematic.
The first one consists of a nonlinear elastic energy which is influenced by the orientation of the molecules of the nematic elastomer.
The nematic energy is an Oseen--Frank energy in the deformed configuration.
The constraint of the positivity of the determinant of the deformation gradient is imposed.
The functionals are not assumed to have the usual polyconvexity or quasiconvexity assumptions to be lower semicontinuous.
We instead compute its relaxation, that is, the lower semicontinuous envelope, which turns out to be the quasiconvexification of the mechanical term plus the tangential quasiconvexification of the nematic term.
The main assumptions are that the quasiconvexification of the mechanical term is polyconvex and that the deformation is in the Sobolev space $W^{1,p}$ (with $p>n-1$ and $n$ the dimension of the space) and does not present cavitation. 
\end{abstract}


\section{Introduction}

Liquid crystal elastomers are hybrid materials that combine the orientational order of liquid crystals with the elastic properties of rubber-like solids.
They are constituted by a network of long, crosslinked polymer chains.
It is this cross-linking what differentiates a liquid crystal elastomer from an ordinary liquid crystal polymer.
In the inner structure of these elastomers, some elongated rigid monomer units (called \emph{mesogens}) are incorporated to the polymer chain.
As any liquid crystal, it can have several phases, according to its internal ordering; they are usually classified in nematic, smectic and cholesteric.
In the nematic phase, which is the study of this work, the molecules self-align to have a long-range directional order.
In fact, most nematic liquid crystals are uniaxial: they have one axis that is longer and preferred.
When we assume that the degree of order is fixed (along space and time), the order can be described by a unit vector field $\vec n$, indicating the preferred axis: this leads to the Oseen--Frank theory.
In fact, if their degree of order is not fixed, then the more elaborated Landau--de Gennes' $Q$-tensor theory is used instead.
Classic references for liquid crystals are \cite{DePr93,Virga94}, and one specifically for liquid crystal elastomers is \cite{WaTe07}.

In the small deformation regime, the director field $\vec n$ can be defined in the reference configuration, but when large deformations are present, it has to be evaluated at points in the deformed configuration (see \cite{DeTe09,BaDe15}).
Thus, while in hyperelasticity \cite{Ball77} one usually assumes that the mechanical energy of a deformation $u : \Om \to \Rn$ is of the form
\begin{equation}\label{eq:W0intro}
 \int_{\Om} W_0 (D u (x)) \, \dd x ,
\end{equation}
(where $\Om \subset \Rn$ represents the body in its reference configuration), the coupling of rubber elasticity with the orientational order of the molecules produces a strong anisotropic behaviour, and the energy is given by
\begin{equation}\label{eq:Wintro}
 I_{\mec} (u, \vn) := \int_{\Om} W (D u (x), \vec n (u (x))) \, \dd x .
\end{equation}
In this way, the energy density not only depends on the deformation gradient $Du$ but it is also influenced by the director field $\vec n$ evaluated in the deformed configuration.
Normally \cite{DeTe09,BaDe15}, given an elastic-energy function $W_0$ and a fixed degree of amplitude $\alpha > 0$, one takes
\begin{equation}\label{eq:WW0}
 W (F, \vec n) = W_0 \left( \left( \alpha^{-1} \vec n \otimes \vec n + \sqrt{\alpha} ( I - \vec n \otimes \vec n) \right) F \right) .
\end{equation}
The material parameter $\alpha$ describes the amount of local distortion, and the tensor $\alpha^{-1} \vec n \otimes \vec n + \sqrt{\alpha} (I - \vec n \otimes \vec n)$ represents a volume-preserving uniaxial stretch of amplitude $\alpha^{-1}$ along the direction $\vec n$; here $I$ denotes the identity matrix.
In this work, however, we allow for a general dependence of $\vec n$, so that $W$ is not necessarily of the form \eqref{eq:WW0}.
In fact, for the sake of generality, in this article we work in dimension $n$, despite the physically relevant case is, of course, $n=3$.

The vector field $\vn$ takes values in the unit sphere $\mS^{n-1}$, although, because of the head-to-tail symmetry of the nematics (i.e., the fact that $\vec n$ is indistinguishable from $- \vn$; see, e.g., \cite{DePr93}), it should take values in the real projective space of dimension $n-1$; see \cite{BaZa11} for a comparison between the two models.
Still, we adopt the more usual approach of $\mS^{n-1}$ and, in order to take into account the head-to-tail symmetry, the energy density $W : \Rnn \times \mS^{n-1} \to [0, \infty]$ of \eqref{eq:Wintro} has to satisfy $W (F, \vec n) = W (F, - \vec n)$ for all $F \in \Rnn$ and $\vn \in \mS^{n-1}$.
It must also meet the principle of objectivity, but in this work we will not use that assumption.

The model that we adopt for the nematic elastomers is, with some small generalizations, that of Barchiesi and DeSi\-mo\-ne \cite{BaDe15} (see also \cite{DeTe09,AgDe11} for earlier studies and \cite{BaHeMo17} for a later slight generalization, which in fact is the starting point of this work).
Accordingly, the energy $I$ associated to the deformation $u$ and the director $\vec n$ is the sum of two contributions: $I = I_{\mec} + I_{\nem}$, where $I_{\mec}$ is as in \eqref{eq:Wintro}, and
\begin{equation}\label{eq:Inem}
 I_{\nem} (u, \vec n) := \int_{u(\Om)} V (\vn (y), D \vec n (y)) \, \dd y .
\end{equation}
The term $I_{\mec}$ is, as explained above, the mechanical energy of the deformation, where the effect of the orientation of the molecules is taken into account.
The term $I_{\nem}$, the nematic energy, is an Oseen--Frank energy in the deformed configuration; we have denoted by $D \vec n$ the gradient of $\vn$.
It is important that the function $V$ that appears in \eqref{eq:Inem} has an explicit dependence on $\vec n$, since the most typical Oseen--Frank energy is (in dimension $3$) of the form
\begin{equation}\label{eq:OF}
 K_1 \left( \Div \vn \right)^2 + K_2 \left( \vn \cdot \curl \vn \right)^2 + K_3 \left| \vn \times \curl \vn \right|^2 + \left( K_2 + K_4 \right) \left( \tr (D \vn)^2 - ( \Div \vn )^2 \right) ,
\end{equation}
for some constants $K_1, \ldots, K_4$, although sometimes the easier particular case
\begin{equation}\label{eq:oneconstant}
 K \left| D \vn \right|^2
\end{equation}
is used, which is the so-called \emph{one-constant approximation} and corresponds to the choice $K_1 = K_2 = K_3$, $K_4 = 0$.
In general, the role of the energy density $V$ is to penalize variations of the nematic director, and, more precisely, the main types of distortion in a nematic: splay, twist and bend.
We recall that, although formula \eqref{eq:OF} is usually applied when $\vn$ is defined in the reference configuration $\Om$, it is also a valid model when $\vn$ is defined in the deformed configuration $u (\Om)$.
In this case, the head-to-tail symmetry requests $V (\vn, G) = V (-\vn, -G)$ for all arguments $(\vn,G)$ where $V$ is defined.

Existence of minimizers for the functional $I$ was proved first in \cite{BaDe15} and then generalized in \cite{BaHeMo17}, for $W$ of the form \eqref{eq:WW0} and $V$ being \eqref{eq:oneconstant}.
In any case, it was clear from the proof that the key hypotheses were the polyconvexity of $W$ and the quasiconvexity of $V$.
These assumptions imply the lower semicontinuity of both functionals $I_{\mec}$ and $I_{\nem}$, and, together with suitable coercivity assumptions, the direct method of the calculus of variations guarantees the existence of minimizers.
The main difficulty in that analysis were the composition $\vn \circ u$ in the term $I_{\mec}$ (since composition is not continuous in general with respect to the weak topology) and the fact that the domain of integration in $I_{\nem}$ depends on $u$.
Those obstacles were overcome by the use of a local invertibility property for the class of deformations $u$ in the admissible set.

In this work we remove the conditions leading to the lower semicontinuity: the function $W$ is not polyconvex (not even quasiconvex) and $V$ is not quasiconvex (in fact, not \emph{tangentially quasiconvex}, which is the natural convexity assumption in this context; see below).
Then, minimizers may not exist, and the usual approach is the computation of a relaxed (or \emph{effective}) energy.
Relaxation typically indicates the formation of microstructure; see, e.g., \cite{BaJa87,Muller99,CoDo14} in the context of elasticity, and \cite{DeDo02,Silhavy07,CeDe11} for nematic elastomers.

Since the result of Dacorogna \cite{Dacorogna82}, we know that under $p$-growth conditions (where $p$ is the exponent of the Sobolev space $W^{1,p}$ where the problem is set), the relaxation of a functional of the form \eqref{eq:W0intro} is
\begin{equation}\label{eq:qcW0}
 \int_{\Om} W_0^{qc} (D u (x)) \, \dd x ,
\end{equation}
where $W_0^{qc}$ is the \emph{quasiconvexification} of $W_0$.
However, a $p$-growth condition is incompatible with the standard assumption in nonlinear elasticity in which it is required that $W_0$ is infinity in matrices $F$ with $\det F \leq 0$ and
\begin{equation}\label{eq:W0infty}
 W_0 (F) \to \infty \quad \text{as} \quad \det F \to 0 .
\end{equation}
Conti and Dolzmann \cite{CoDo15} have recently proved the first relaxation result for energies $W_0$ satisfying \eqref{eq:W0infty}.
The conclusion is that \eqref{eq:qcW0} is indeed the relaxation of \eqref{eq:W0intro}, whereas the main assumptions are that $W_0^{qc}$ is polyconvex, and that the exponent $p$ of the Sobolev space where the problem is set satisfies $p \geq n$.

When $\Om' \subset \Rn$ is a fixed domain, the relaxation of an energy of the form
\[
 \int_{\Om'} V (D \vec n (y)) \, \dd y
\]
when $\vn$ takes values in the unit sphere (or, in general, in a manifold) was proved in Dacorogna \emph{et al.}\ \cite{DaFoMaTr99} to be 
\[
 \int_{\Om'} V^{tqc} (D \vec n (y)) \, \dd y ,
\]
where $V^{tqc}$ is the \emph{tangential quasiconvexification} of $V$ (see Section \ref{se:quasiconvex} for the definition).
In our case, however, the domain of integration $u (\Om)$ in $I_{\nem}$ varies along the minimizing sequence or the test functions, so the result of \cite{DaFoMaTr99} is not directly applicable.
Our function $V$ also has an extra dependence on $\vn$, but this is not a problem because it is a lower-order perturbation (see \cite{AlLe01}).

Finally, it is immediate to see from the definition that the relaxation of a sum is less than or equal to the sum of the relaxations, so knowing the relaxation of each term $I_{\mec}$ and $I_{\nem}$ is insufficient to compute the relaxation of $I$, unless we have an extra condition implying that the two processes of relaxation do not interfere.

In this paper we prove that the relaxation of $I$ is
\begin{align*}
 & I^* := I^*_{\mec} + I^*_{\nem} , \text{ with} \\
 & I^*_{\mec} (u, \vn) := \int_{\Omega} W^{qc} (Du(x),\vn(u(x))) \, \dd x , \qquad I^*_{\nem} (u, \vn) :=  \int_{u (\Omega)} V^{tqc} (\vn(y),D\vn(y)) \, \dd y ,
\end{align*}
where $W^{qc}$ is the quasiconvexification of $W$ in the first variable and, as in \cite{CoDo15}, $W^{qc}$ is assumed to be polyconvex.
The exponent $p$ of the Sobolev space where $u$ lies satisfies $p > n-1$, which constitutes an improvement of the result of \cite{CoDo15}.
In the next paragraphs we comment on the main ideas of the proof.

A relaxation result is usually proved in two steps: a lower bound and an upper bound.
The lower bound inequality consists in proving that the functional $I^*$ is lower semicontinuous, and the proof of this fact is a slight generalization of that of \cite{BaHeMo17}.
Hence, the bulk of the proof of the relaxation result relies, as in \cite{CoDo15}, in the upper bound, which amounts to the construction of a \emph{recovery sequence}: for each $(u, \vn)$ we must find a sequence $\{ (u_j, \vn_j) \}_{j \in \N}$ such that $u_j \to u$ in $L^1 (\Om, \Rn)$, $\vn_j \to \vn$ in $L^1$ (in a precise sense, since the domain of definition of each $\vn_j$ varies) and $I(u_j, \vn_j) \to I^* (u, \vn)$ as $j \to \infty$.

We start with the term $I_{\mec}$.
We recall from \cite{CoDo15} that the reason to choose $u$ to be in the Sobolev space $W^{1,p}$ with $p \geq n$ is because this space makes the determinant of the gradient weakly continuous in $L^1$, i.e., if $u_j \weakc u$ as $j \to \infty$ in $W^{1,p}$ with $\det Du_j > 0$ a.e.\ for all $j \in \N$ then $\det Du_j \weakc \det D u$ in $L^1$.
Functions in $W^{1,p}$ with $p \geq n$ also enjoy nice properties such as the continuity (this is Morrey's \cite{Morrey08} embedding theorem for $p >n$ and was proved in \cite{VoGo76} for $p=n$ under the assumption $\det Du >0$ a.e.).
Nevertheless, there is a large amount of work about the continuity of the determinant in the space $W^{1,p}$ with $p>n-1$, as well as extra regularity properties of such functions, provided some additional conditions hold; see \cite{Ball77,Sverak88,MuQiYa94,MuSp95,HeMo10,HeMo12}.
In fact, the possibility of lowering the exponent from $p \geq n$ to $p>n-1$ by using those results was already suggested in \cite{CoDo15}.
Here we took the tools from Barchiesi \emph{et al.}\ \cite{BaHeMo17}, where it was defined a class $\Ap$ of functions $u \in W^{1,p}$ ($p>n-1$) with $\det D u >0$ a.e.\ such that, in a precise way, no \emph{cavitation} occurs (cavitation is the formation of voids in the material, see \cite{MuSp95}).
This class contains the familiar classes $\A_{p,q}$, studied in \cite{Ball77,Sverak88,MuQiYa94}, formed by the Sobolev maps $u$ in $W^{1,p}$ such that $\cof Du \in L^q$ and $\det Du>0$ a.e., for $p>n-1$ and $q\geq \frac{n}{n-1}$.
It was proved in \cite{BaHeMo17} that many properties that $W^{1,n}$ enjoy also hold in $\Ap$.
The most important ones for this work are the weak continuity of the determinant and the local invertibility, which states that for a.e.\ $x \in \Om$ there is $r>0$ such that $u$ is invertible in $B(x, r)$.
This local invertibility property is the key to analyzing functionals like $I$ that involve both reference and deformed configurations.
Thus, the recovery sequence $\{ u_j \}_{j \in \N}$ for $u$ and, hence, the treatment of the term $I_{\mec}$ is an adaptation of the construction of \cite{CoDo15} but using some tools of \cite{BaHeMo17}.
As a direct corollary of our study we obtain that the relaxation result of \cite{CoDo15} can be extended to the functions in the class $\A_p$ (choosing $W$ not depending on $\vn$ and $V=0$, even though $V=0$ does not satisfy our assumptions).
They key idea is to modify the value of a given $u$ in balls, so that in those balls $u$ is replaced by a certain composition $u \circ v$ in such a way that the orientation-preserving condition remains and that the modified function still belongs to $\Ap$.
Moreover, the image of $u$ coincides with the image of the modified function.
In this way, we construct a sequence $\{ u_j \}_{j\in\N}$ in $\Ap$ such that $u_j (\Om) = u (\Om)$ for all $j \in \N$, $u_j \to u$ in $L^1 (\Om, \Rn)$ and $I_{\mec} (u_j, \vn) \to I^*_{\mec} (u, \vn)$ as $j \to \infty$.
At this point, we ought to mention that the image $u(\Om)$ requires a precise definition, since $u$ is, in principle defined a.e., and $u(\Om)$ must be open so that $\vn$ is in the Sobolev space $W^{1,s} (u(\Om), \mS^{n-1})$.
These technicalities were solved in \cite{BaHeMo17}.

The term $I_{\nem}$ is tackled as in \cite{DaFoMaTr99} with the use of the tangential convexification.
In principle, the only obstruction to apply their result directly is that the domain $u(\Om)$ may vary along the recovery sequence, but, as explained in the previous paragraph, the recovery sequence $\{u_j\}_{j \in \N}$ constructed for $u$ satisfies that $u_j (\Om) = u (\Om)$.
This equality is also the reason why the two processes of relaxation do not interfere and we have that the relaxation of $I$ is the sum of the relaxations, i.e., $I^* = I^*_{\mec} + I^*_{\nem}$.

Although the motivation of this work is the model for nematic elastomers explained above, the techniques presented here should be useful for other models involving reference and deformed configurations, like those in magnetoelasticity (see \cite{RyLu05,KrStZe15,BaHeMo17}) or the Landau--de Gennes model for liquid crystal elastomers (see \cite{CaGaYa15,BaHeMo17}).
In this respect, this work seems to be the first study where the relaxation in the deformed configuration has been performed.

The article is structured as follows.
Section \ref{sect:caprelaxdef} establishes the definitions and notations used throughout the paper.
Section \ref{se:quasiconvex} reviews the concepts of polyconvexity, quasiconvexity and tangential quasiconvexity.
In Section \ref{sect:caprelaxenergy} we define the class $\Ap$ and recall some results from \cite{BaHeMo17} that will be used in the paper.
We also show some new results in the class $\Ap$ in order to prove that the recovery sequence to be constructed in Section \ref{sect:caprelaxrel} is indeed in $\Ap$.
Section \ref{se:existence} proves the lower bound inequality, as well as the existence of minimizers for $I^*$.
In Section \ref{se:productchain} we recall three auxiliary results from \cite{CoDo15} about the product of $L^1$ functions and the chain rule for Sobolev functions.
Section \ref{sect:caprelaxrel} is the core of the paper: we prove the upper bound inequality by the construction of a recovery sequence.
The paper finishes with Section \ref{se:generalizations}, where the relaxation result is established as a consequence of the results of Sections \ref{se:existence} and \ref{sect:caprelaxrel}.

\section{General notation}\label{sect:caprelaxdef}

In this section we establish the general notation and definitions used in the paper.
We postpone the definitions regarding the class $\Ap$ to Section \ref{sect:caprelaxenergy}.

We will work in dimension $n \geq 2$.
In all the paper, $\Om$ is a non-empty bounded open set of $\R^n$, which represents the body in its reference configuration.

The closure of a set $A$ is denoted by $\bar{A}$ and its boundary by $\partial A$.
Given two sets $U,V$ of $\Rn$, we will write $U \ssubset V$ if $U$ is bounded and $\bar{U} \subset V$.
The open ball of radius $r>0$ centred at $x \in \Rn$ is denoted by $B(x, r)$.

Given a square matrix $A \in \Rnn$, its determinant is denoted by $\det A$.
The adjugate matrix $\adj A \in \Rnn$ satisfies $(\det A) I = A \adj A$, where $I$ denotes the identity matrix.
The transpose of $\adj A$ is the cofactor $\cof A$.
If $A$ is invertible, its inverse is denoted by $A^{-1}$.
The inner (dot) product of vectors and of matrices will be denoted by $\cdot$ and their associated norms are denoted by $\left| \cdot \right|$.
Given $a, b \in \Rn$, the tensor product $a \otimes b$ is the $n \times n$ matrix whose component $(i,j)$ is $a_i \, b_j$.
The set $\Rnn_+$ denotes the subset of matrices in $\Rnn$ with positive determinant, while $SL(n) \subset \Rnn$ is the set of matrices with determinant one.
The set $\mS^{n-1}$ denotes the subset of unit vectors in $\Rn$.

The symbol $\lesssim$ is used to indicate that the quantity of the left-hand side is less than or equal to a positive constant (whose precise value is not important) times the right-hand side.
This constant is, of course, independent of the main quantity to estimate, which should be clear from the context.

The Lebesgue measure in $\Rn$ is denoted by $\left| \cdot \right|$, and the $(n-1)$-dimensional Hausdorff measure by $\Ha^{n-1}$.
For $1 \leq p \leq \infty$, the Lebesgue $L^p$ and Sobolev $W^{1,p}$ spaces are defined in the usual way.
So are the functions of class $C^k$, for $k$ a positive integer of infinity, and their versions $C^k_c$ of compact support.
The derivative of a Sobolev or $C^k$ function $u$ is written $D u$.
The conjugate exponent of $p$ is $p'$.
We will indicate the domain and target space, as in, for example, $L^p (\Om,\Rn)$, except if the target space is $\R$, in which case we will simply write $L^p (\Om)$; the corresponding norm is written $\left\| \cdot \right\|_{L^p (\Om,\Rn)}$.
Given $S \subset \Rn$, the space $L^p (\Om ,S)$ denotes the set of $u \in L^p (\Om,\Rn)$ such that $u (x) \in S$ for a.e.\ $x \in \Om$, and analogously for other function spaces.
Weak convergence in $L^p$ or $W^{1,p}$ is indicated by $\weakc$, while $\weakcs$ is the symbol for weak$^*$ convergence in $L^{\infty}$.
Strong or a.e.\ convergence is denoted by $\to$.
Given a measurable set $A$ the symbol $\fint_A$ denotes the integral in $A$ divided by the measure of $A$.
The identity function in $\Rn$ is denoted by $\id$.

\section{Polyconvexity, quasiconvexity and tangential quasiconvexity}\label{se:quasiconvex}

Quasiconvexity is a central concept in the calculus of variations, since, under suitable growth assumptions, it is necessary and sufficient for the lower semicontinuity of functionals of the form \eqref{eq:W0intro} in the weak topology of $W^{1,p}$ (see the pioneering results of \cite{Morrey52,AcFu84} or the monograph \cite{Dacorogna08}).
However, no lower semicontinuity results have been proved so far for quasiconvex integrands $W_0$ satisfying \eqref{eq:W0infty}.
Here is where the concept of \emph{polyconvexity} comes into play (see, e.g., \cite{Ball77,BaCuOl81,Dacorogna08}).
Let $\tau$ be the number of minors of an $n \times n$ matrix; we call $\R^{\tau}_+:=\R^{\tau-1} \times (0, \infty)$ and denote by $M(F) \in \R^{\tau}$ the collection of all the minors of an $F \in \Rnn$ in a given order such that its last component is $\det F$; we denote by $M_0(F) \in \R^{\tau-1}$ the collection of all the minors of an $F \in \Rnn$ except the determinant, in a given order.
For the sake of clarity, in the following definition of polyconvexity, we single out three cases, according to whether the domain of definition is the set of all matrices or only those with positive determinant or only those with determinant one.

\begin{definition}
\begin{enumerate}[a)]
\item
A Borel function $W_0 : SL(n) \to \R \cup \{\infty\}$ is polyconvex if there exists a convex function $\Phi : \R^{\tau-1} \to \R \cup \{\infty\}$ such that $W_0 (F) = \Phi (M_0(F))$ for all $F \in SL(n)$.

\item
A Borel function $W_0 : \Rnn_+ \to \R \cup \{\infty\}$ is polyconvex if there exists a convex function $\Phi : \R^{\tau}_+ \to \R \cup \{\infty\}$ such that $W_0 (F) = \Phi (M(F))$ for all $F \in \Rnn_+$.

\item
A Borel function $W_0 : \Rnn \to \R \cup \{\infty\}$ is polyconvex if there exists a convex function $\Phi : \R^{\tau} \to \R \cup \{\infty\}$ such that $W_0 (F) = \Phi (M(F))$ for all $F \in \Rnn$.
\end{enumerate}
\end{definition}

We remark that if a $W_0 : SL(n) \to \R \cup \{\infty\}$ or $W_0 : \Rnn_+ \to \R \cup \{\infty\}$ is polyconvex, then its extension by infinity to $\Rnn$ is also polyconvex.

In our study, we will deal with functions $W$ with values in $\R \cup \{\infty\}$ defined in $SL(n)\times \mS^{n-1}$, $\Rnn_+ \times \mS^{n-1}$ or $\Rnn \times \mS^{n-1}$.
We will say that they are polyconvex in the first variable (or, in short, polyconvex) if $W (\cdot, \vn)$ is polyconvex for all $\vn \in \mS^{n-1}$.

We now recall the classical concept of quasiconvexity.
Its definition is done so that the function can take infinite values (see, e.g., \cite{BaMu84}).

\begin{definition}\label{de:quasiconvex}
A Borel function $W_0 : \Rnn \to \R \cup \{\infty\}$ is quasiconvex if for all $F \in \Rnn$ and all $\vp \in W^{1, \infty} (B(0,1), \Rn)$ with $\vp (x) = Fx$ on $\partial B(0,1)$, we have
\[
 W_0(F) \leq \fint_{B(0,1)} W_0(D\vp) \, \dd x .
\]
\end{definition}

The equality $\vp (x) = Fx$ on $\partial B(0,1)$ is understood in the sense of traces.
A Borel function $W_0 : SL(n) \to \R \cup \{\infty\}$ or $W_0 : \Rnn_+ \to \R \cup \{\infty\}$ is quasiconvex if its extension by infinity is quasiconvex.

When $W$ takes always finite values, there are some possible equivalent definitions of its \emph{quasiconvexification} (see, e.g., \cite{Dacorogna08}), but when $W$ is infinity in some parts of its domain, the definitions are no longer equivalent.
We adopt that of \cite{CoDo15}, which is the natural one corresponding to Definition \ref{de:quasiconvex} and reads as follows.

\begin{definition}\label{de:quasiconvexification}
The quasiconvexification $W_0^{qc}: \Rnn \to \R \cup \{\infty\}$ of a Borel function $W : \Rnn \to \R \cup \{\infty\}$
is defined as
\[
 W_0^{qc} (F) := \inf \left\{\fint_{B(0,1)} W_0(D\vp) \, \dd x: \, \vp\in W^{1,\infty}(B(0,1),\R^{n}), \, \vp(x)=Fx \text{ on } \partial B(0,1)\right\} .
\] 
\end{definition}



For functions $W : \Rnn \times \mS^{n-1} \to \R \cup \{\infty\}$, its quasiconvexification $W^{qc}$ refers to the first variable.
It is well known that a finite-valued quasiconvex function is rank-one convex; in particular, it is continuous.
When the function takes infinite values, this fact was proved in \cite{Fonseca88}.
For functions $W : \Rnn \times \mS^{n-1} \to \R \cup \{\infty\}$, the corresponding continuity result is as follows.

\begin{proposition}\label{pr:Wqccont}
Assume that $W: \R^{n\times n}_{+} \times \mS^{n-1} \to [0,\infty)$ is continuous and there exists an $h:[0,2] \to [0,\infty)$ with $\lim_{t\to 0} h(t)=0$ such that for all $F \in \Rnn_+$ and $\vn,\vm \in \mS^{n-1}$,
\begin{equation}\label{eq:modcontW}
 \left| W(F,\vn)-W(F,\vm) \right|\le h \left( |\vn-\vm| \right) W(F,\vn)  .
\end{equation}
Extend $W$ by infinity outside $\R^{n\times n}_{+} \times \mS^{n-1}$.
Then $W^{qc}|_{\R^{n\times n}_{+} \times \mS^{n-1}}$ is continuous.
\end{proposition}
\begin{proof}
First we prove that for each $G \in \R^{n\times n}_{+}$ there exists $M_G>0$ such that for all $\vn , \vl \in \mS^{n-1}$,
\begin{equation}\label{eq:Wqccont}
 \left| W^{qc}(G,\vn)-W^{qc}(G,\vl) \right| \le M_G \, h(|\vn-\vec{\ell}|) .
\end{equation}
Indeed, fix $\ve>0$ and for each $\vm\in\mS^{n-1}$ let $\psi_{\vm} \in W^{1,\infty}(B(0,1),\R^{n})$ be such that $\psi (x)=Gx$ on $\partial B(0,1)$ and
\[
 \fint_{B(0,1)}W(D\psi_{\vm},\vm)\, \dd x \leq W^{qc}(G,\vm) + \ve.
\]
Define $M_G=\sup_{\vm\in\mS^{n-1}}W(G,\vm)$, which satisfies $M_G < \infty$ thanks to the continuity of $W$.
Moreover, for each $\vm\in\mS^{n-1}$,
\[
 W^{qc}(G,\vm) \le W(G,\vm) \le M_G ,
\]
so
\[
 \sup_{\vm\in\mS^{n-1}} \fint_{B(0,1)}W(D\psi_{\vm},\vm)\, \dd x \leq M_G + \ve.
\]
Now, for all $\vn, \vl \in\mS^{n-1}$,
\begin{align*}
 & W^{qc}(G,\vn)-W^{qc}(G,\vl) \le \fint_{B(0,1)} \left[ W(D\psi_{\vl},\vn)-W(D\psi_{\vl},\vl) \right] \dd x+\ve\\
&\le h(|\vn-\vec{\ell}|)\fint_{B(0,1)}W(D\psi_{\vl},\vl)\, \dd x+ \ve \le h (|\vn-\vec{\ell}|) \left(M_G+\ve\right) + \ve.
\end{align*}
As this is true for all $\ve>0$ we obtain
\[
 W^{qc}(G,\vn)-W^{qc}(G,\vl) \le M_G \, h (|\vn-\vec{\ell}|) ,
\]
and, by the symmetry of the argument we conclude \eqref{eq:Wqccont}.

Now let $F \in \R^{n\times n}_{+}$ and $\vl \in \mS^{n-1}$ and fix $\ve>0$.
By \cite[Th.\ 2.4 and Prop.\ 2.3]{Fonseca88}, $W^{qc} (\cdot, \vl)$ is continuous.
Therefore, there exists $\delta>0$ such that if $G \in \R^{n\times n}_{+}$ satisfies $|G - F| \leq \delta$ then
\[
 \left| W^{qc} (G, \vl) - W^{qc} (F, \vl) \right| \leq \ve ,
\]
so for all $\vn \in \mS^{n-1}$ we have, using \eqref{eq:Wqccont} and the triangle inequality,
\[
 \left| W^{qc} (G, \vn) - W^{qc} (F, \vl) \right| \leq M_G \, h(|\vn-\vec{\ell}|) + \ve \leq M_{F,\delta} \, h(|\vn-\vec{\ell}|) + \ve,
\]
where $M_{F,\delta} := \sup \left\{ M_G : \, G \in \Rnn_+ , \, |G - F| \leq \delta \right\}$, which is finite because of the continuity of $W$.
This concludes the proof.
\end{proof}

The proof under incompressibility is analogous and will be omitted.
Its statement is as follows.
\begin{proposition}
Assume that $W: SL(n) \times \mS^{n-1} \to [0,\infty)$ is continuous and there exists an $h:[0,2] \to [0,\infty)$ with $\lim_{t\to 0} h(t)=0$ such that for all $F \in SL(n)$ and $\vn,\vm \in \mS^{n-1}$, inequality \eqref{eq:modcontW} holds.
Extend $W$ by infinity outside $SL(n) \times \mS^{n-1}$.
Then $W^{qc}|_{SL(n) \times \mS^{n-1}}$ is continuous.
\end{proposition}

We now explain the concept of \emph{tangential quasiconvexity} and \emph{tangential quasiconvexification}.
For this, we fix a $C^{1}$ manifold $\mathcal{M}$ embedded in $\R^{n}$ (although we will always take $\mathcal{M} = \mS^{n-1}$); all concepts of \emph{tangential} are referred to the manifold $\mathcal{M}$.
For each $z \in \mathcal{M}$ we denote the tangent space of $\mathcal{M}$ at $z$ by $T_z \mathcal{M}$.
Given a Sobolev function $\vn$ defined in an open set $U \subset \R^n$ such that $\vn (y) \in \mathcal{M}$ for a.e.\ $y \in U$, we have that $D \vn (y) \in (T_{\vn(y)} \mathcal{M})^n$ for a.e.\ $y \in U$.
Therefore, the function $V$ of \eqref{eq:Inem} need only be defined in
\[
 T^n \mathcal{M} := \left\{ (z, \zeta) : \, z \in \mathcal{M}, \, \zeta \in (T_z \mathcal{M})^n \right\} .
\]
Thus, we consider a Borel function $V : T^n \mathcal{M} \to [0, \infty)$.
The following definition is due to Dacorogna \emph{et al.}\ \cite{DaFoMaTr99} when $V$ does not depend on the first variable.
The natural definition for a $V$ defined in the whole $T^n \mathcal{M}$ is straightforward (see \cite{AlLe01}).

\begin{definition}
\begin{enumerate}[a)]
Let $V : T^n \mathcal{M} \to [0, \infty)$ be a Borel function.
\item
$V$ is tangentially quasiconvex if for all $(z, \zeta) \in T^n \mathcal{M}$ and all $\vp \in W^{1,\infty} (B(0,1), T_z\mathcal{M})$ with $\vp (y) = \zeta y$ on $\partial B (0,1)$ we have
\[
 V (z,\zeta) \leq \fint_{B(0,1)} V(z, D\vp(y))\, \dd y .
\]

\item
The tangential quasiconvexification $V^{tqc} : T^n \mathcal{M} \to [0, \infty)$ of $V$ is
\begin{align*}
 & V^{tqc} (z,\zeta) \\
 & := \inf \left\{ \fint_{B(0,1)} \! V(z, D\vp(y))\, \dd y : \, \vp\in W^{1,\infty} (B(0,1), T_z\mathcal{M}) , \, \vp (y) = \zeta y \text{ on } \partial B (0,1) \right\}.
\end{align*}
\end{enumerate}
\end{definition}

The equality $\vp (y) = \zeta y$ on $\partial B (0,1)$ is understood in the sense of traces and we are regarding $\zeta$ as an $n \times n$ matrix.
Note that the fact $\vp \in W^{1,\infty} (B(0,1), T_z\mathcal{M})$ implies $D\vp (y) \in (T_z\mathcal{M})^n$ for a.e.\ $y \in B(0,1)$.
Standard arguments (see, e.g., \cite[Prop.\ 5.11]{Dacorogna08}) show that the choice of $B(0,1)$ as domain of integration is irrelevant.

From the definitions, it is immediate to check that $V^{tqc}$ is tangentially quasiconvex and that $V$ is tangentially quasiconvex if and only if $V = V^{tqc}$.

The next proposition and theorem summarize the main results of \cite{DaFoMaTr99}; again, the formulation is adapted to cover a dependence of $V$ on the first variable as well (see \cite{AlLe01}).

\begin{proposition}\label{prop:tqc}
\begin{enumerate}[a)]
\item
For each $z \in \mathcal{M}$, let $P_z \in \Rnn$ be the matrix corresponding to the orthogonal projection from $\Rn$ onto $T_z\mathcal{M}$.
Define $\bar{V} : \mathcal{M} \times \Rnn \to [0, \infty)$ as
\[
 \bar{V} (z, \zeta) := V(z, P_z \zeta)
\]
and let $\bar{V}^{qc}$ be the quasiconvexification of $\bar{V}$ with respect to the second variable.
Then $V^{tqc} = \bar{V}^{qc}|_{T^n \mathcal{M}}$.

\item Let $\mathcal{M} = \mS^{n-1}$.
Define $\bar{V} : \mS^{n-1} \times \Rnn \to [0, \infty)$ as
\[
 \bar{V} (z, \zeta) := V \left( z, (I - z \otimes z) \zeta \right)
\]
and let $\bar{V}^{qc}$ be the quasiconvexification of $\bar{V}$ with respect to the second variable.
Then $V^{tqc} = \bar{V}^{qc}|_{T^n \mS^{n-1}}$.

\end{enumerate}
\end{proposition}

\begin{theorem}\label{th:tqc}
Let $\Om' \subset \R^n$ be open and bounded.
Let $s \geq 1$.
Let $V : T^n \mathcal{M} \to [0, \infty)$ be continuous and satisfy
\[
 V (z, \zeta) \leq C \left( 1 + \left| \zeta \right|^s \right) ,Ê\qquad (z, \zeta) \in T^n \mathcal{M} 
\]
for some $C > 0$.
Let $\vn \in W^{1,s} (\Om', \mathcal{M})$.
The following hold:

\begin{enumerate}[a)]
\item 
If $V$ is tangentially quasiconvex then, for any sequence $\{ \vn_j \}_{j \in \N} \subset W^{1,s} (\Om', \mathcal{M})$ converging weakly to $\vn$ in $W^{1,s} (\Om', \mathcal{M})$, we have
\[
 \int_{\Om'} V (\vn (y), D \vn (y)) \, \dd y \leq \liminf_{j \to \infty} \int_{\Om'} V (\vn_j (y), D \vn_j (y)) \, \dd y .
\]

\item 
$\displaystyle
 \inf \left\{ \liminf_{j \to \infty} \int_{\Om} V (\vn_j (y), D \vn_j (y)) \, \dd y : \vn_j \weakc \vn \text{ in } W^{1,s} (\Om', \mathcal{M}) \right\} = \int_{\Om'} V^{tqc} (\vn (y), D \vn (y)) \, \dd y .
$
\end{enumerate}
\end{theorem}

As commented in \cite{Mucci09}, using Proposition \ref{prop:tqc}, we find that $V$ is tangentially quasiconvex if and only if it is the restriction of a quasiconvex function (in the second variable) $\bar{V} : \mathcal{M} \times \Rnn \to [0,\infty)$.
Since finite-valued quasiconvex functions are continuous (because they are rank-one convex), we infer that any tangentially quasiconvex $V : T^n \mathcal{M} \to [0, \infty)$ is continuous in the second variable.

\section{Class $\Ap$}
\label{sect:caprelaxenergy}

In this section we define the class $\Ap$ of functions that will be the object of this work.
Its main aim is to present the results showing that, similarly to what occurs in Sobolev spaces, under some additional conditions the cut-and-paste of functions in the class $\Ap$ is still in the class $\Ap$ (Lemma \ref{lem paste Apq}) and the composition of an orientation-preserving Lipschitz function with a function of class $\Ap$ is still in $\Ap$ (Lemma \ref{lem inner comp by lip}).
The reader not interested in the technicalities of the class $\Ap$ may omit this section and admit Lemmas \ref{lem paste Apq} and \ref{lem inner comp by lip}.

The class $\Ap$ consists, roughly, in the set of $u \in W^{1, p} (\Om, \Rn)$ such that $\det D u > 0$ a.e.\ and no \emph{cavitation} occurs.
Cavitation is the formation of voids in some materials in extension (see \cite{GeLi59} for the physical process and \cite{MuSp95,SiSp00,CoDe03,HeMo10,HeMo11,HeMo12} for some mathematical developments).
The class $\Ap$ was originally defined in M\"uller \cite{Muller88}, then used by Giaquinta \emph{et al.}\ \cite{GiMoSo98I}, and in Barchiesi \emph{et al.}\ \cite{BaHeMo17} it was proved the local invertibility and extra regularity properties.

This section consists of two subsections.
In Subsection \ref{subse:definitions} we define the class $\Ap$, together with many associated concepts, and state the known results that will be useful in Subsection \ref{subse:results}, where we prove the new results needed for the construction of the recovery sequence in Section \ref{sect:caprelaxrel}.

\subsection{Definitions and previous results}\label{subse:definitions}

This subsection presents the definition of $\Ap$ and its related concepts.
It also states the results that are useful in Subsection \ref{subse:results} in order to prove Lemmas \ref{lem paste Apq} and \ref{lem inner comp by lip}.

\begin{defn}\label{de:inj ae}
A function $u:\Omega\to \R^{n}$ is said to be injective a.e.\ in a subset $A$ of $\Omega$ if there exists a set $N \subset A$ such that $|N|=0$ and $u|_{A\setminus N}$ is injective.
 \end{defn}

We will use the following result.

\begin{proposition}\label{prop:O0}
Given $u \in W^{1,1} (\Om, \Rn)$ with $\det D u > 0$ a.e., there exists a measurable set $\Om_0 \subset \Om$ with $| \Om \setminus \Om_0| =0$ such that:
\begin{enumerate}[a)]
\item\label{item:O0a} $u|_{\Om_0}$ satisfies the change of variables formula.

\item\label{item:O0b} If for some $A \subset \Om$ the restriction $u|_A$ is injective a.e., then $u|_{A\cap \Om_0}$ is injective.
\end{enumerate}
\end{proposition}
Part \emph{\ref{item:O0a})} is due to \cite{Hajlasz93} (see also \cite[Prop.\ 2.6]{MuSp95}).
Part \emph{\ref{item:O0b})} is due to \cite[Lemma 3]{HeMo11}.
The set $\Om_0$ is not uniquely defined; it can be given a precise definition (see, \cite{MuSp95,CoDe03,HeMo11}) but this is not important in the sequel: given a $u$ we just fix any such $\Om_0$.

For any measurable set $A$ of $\Omega$, we define the geometric image of $A$ under $u$ as $u(A\cap \Omega_0)$, and we denote it by $\imG(u,A)$.

We will use the topological degree for continuous functions (see, e.g., \cite{Deimling85,FoGa95b}): if $U\subset \R^{n}$ is a bounded open set, $u: \bar{U}\to\R^{n}$ is continuous and $y\in\R^{n}\setminus u(\partial U)$, we denote by $\deg (u,U,y)$ the degree of $u$ in $U$ at $y$.
If $u:\partial U\to\R^{n}$ is continuous, its degree $\deg (u,U,\cdot)$ is defined as the degree of any continuous extension $\bar{u}:\bar{U}\to\R^{n}$, which exists thanks to Tietze's theorem and does not depend on the extension due to the homotopy-invariance of the degree (see, e.g., \cite[Th.\ 3.1.(d6)]{Deimling85}, \cite[Th.\ 2.4]{FoGa95b}).
If $u\in W^{1,p}(\partial U, \R^{n})$ with $p>n-1$, by Morrey's embedding, $u$ has a continuous representative. 
We define the degree of $u$ in $U$, written $\deg(u,U,\cdot)$, as the degree of its continuous representative.

Now, if $u \in W^{1,p} (\Om, \Rn)$ we define $u^*$ as its precise representative (see, e.g., \cite{Ziemer89}):
\[
 u^* (x) := \lim_{r \to 0} \fint_{B(x,r)} u(z) \, \dd z , 
\]
if that limit exists, and $u^*$ is undefined elsewhere.
It is well known that the above limit exists except on a set of $p$-capacity zero.

Next, we define the topological image (introduced by \v{S}ver\'{a}k \cite{Sverak88}; see also \cite{MuSp95}).
\begin{defn}\label{de:im top}
Let $u\in W^{1,p}(\Omega, \R^n)$ with $p>n-1$.
\begin{enumerate}[a)]
\item Given an open $U \ssubset \Om$ such that $u^* \in W^{1,p}(\partial U, \R^{n})$, we define $\imT(u, U)$, the topological image of $U$ under $u$, as the set of $y\in\R^{n}\setminus u(\partial U)$ such that $\deg (u^*,U,y)\neq 0$. 

\item We define $\imT(u,\Omega)$, the topological image of $\Om$ under $u$, as the union of $\imT(u,U)$ when $U$ runs over all open $U \ssubset \Om$ such that $u^* \in W^{1,p}(\partial U, \R^{n})$.
\end{enumerate}
\end{defn}

Thanks to the continuity of the topological degree for continuous functions we have that $\imT(u,U)$ is an open set, and so is $\imT(u,\Omega)$, as a union of open sets.
Moreover,
\[\imT(u,\Omega)=\bigcup_{i\in\N}\imT(u,U_i)\]
 for every family $\{U_i\}_{i\in\N}$ such that $\Omega=\bigcup_{i\in\N}U_i$, $U_i \ssubset \Om$ and $u^* \in W^{1,p}(\partial U_i, \R^{n})$.

\begin{defn}\label{de:energy}
Let $u \in W^{1,1} (\Omega, \R^{n})$ and $q \geq 1$.
Suppose that $\det D u\in L^{1}\left(\Omega\right)$ and $\cof Du\in L^{q}(\Omega,\R^{n\times n})$.
For $\phi\in W^{1,q'}(\Omega)\cap L^{\infty}(\Omega)$ and $g\in C^{1}_{c}(\R^{n},\R^{n})$ define
\[\E_{\Omega}(u,\phi,g):=\int_{\Omega}\left[\cof Du(x)\cdot (g(u(x))\otimes D\phi(x))+\det D u(x) \, \phi(x)\Div g(u(x))\right] \dd x.\]
\end{defn}

Now we present the class of functions with which we will work in the rest of the chapter.
\begin{defn}
For each $p>n-1$ and $q\ge 1$, we define $\A_{p,q}(\Omega)$ as the set of $u\in W^{1,p}\left(\Omega,\R^{n}\right)$, such that $\det Du\in L^{1} \left(\Omega\right)$, $\cof Du\in L^{q}(\Omega, \Rnn)$, $\det Du>0$ a.e.\ and
\begin{equation}\label{eq:E0}
 \E_{\Omega}(u,\phi,g) = 0, \quad \text{for all } \phi\in C^{1}_c(\Omega) \text{ and } g\in C^{1}_{c}(\R^{n},\R^{n}) .
\end{equation}
We define $\A_{p}(\Omega)=\A_{p,1}(\Omega)$.
We denote by $\A_{p}^{1}(\Omega)$ the set of functions $u\in\A_{p}(\Omega)$ that satisfy $\det Du=1$ a.e.
\end{defn}

If the domain $\Om$ is clear from the context, we will sometimes abbreviate the notation to $\A_{p,q}$, $\A_{p}$ and $\A_{p}^{1}$.

Observe that  $u\in W^{1,p}$ implies $\cof Du\in L^{\frac{p}{n-1}}$, so $\A_{p}(\Omega)=\A_{p,t}(\Omega)$ for $t\in [1,\frac{p}{n-1}]$. 
Moreover, thanks to the result of \cite{MuQiYa94} we have that if $u\in W^{1,p}$ satisfies $\cof Du\in L^{q}$ and $\det Du>0$ a.e.\ with $p>n-1$ and $q\ge \frac{n}{n-1}$ then $u\in \A_{p,q}$.

The following local invertibility result is a particular case of \cite[Cor.\ 4.7]{BaHeMo17}.

\begin{proposition}
Let $u\in\A_{p}(\Omega)$. Then, for a.e.\ $x\in\Omega$ there exists $r>0$ such that $u$ is injective a.e.\ in $B(x,r)$.
\end{proposition}

If $u$ is injective a.e.\ in some $U\ssubset \Om$ then, by Proposition \ref{prop:O0}, $u$ is injective in $U\cap\Omega_0$. 
Therefore $u:U\cap\Omega_0\to \imG(u,U)$ is a bijection.
If, in addition, $u^* \in W^{1,p} (\partial U, \Rn)$ then, thanks to \cite[Th.\ 4.1]{BaHeMo17},
\[ \left| \imT(u,U)\setminus\imG(u,U) \right| = \left| \imG(u,U)\setminus\imT(u,U) \right|=0\]
and, hence, the next definition of local inverse of a function in the class $\A_p$ makes sense.

\begin{defn}\label{de:inv geom}
Let $u\in\A_p(\Omega)$ and $U\ssubset \Om$ be such that $u$ is injective a.e.\ in $U$ and $u^* \in W^{1,p} (\partial U, \Rn)$.
The inverse $(u|_U)^{-1}:\imT(u,U)\to \R^{n}$ is defined a.e.\ as $(u|_{U})^{-1}(y)=x$, for each $y\in \imG(u,U)$, and where $x\in U\cap\Omega_0$ satisfies $u(x)=y$.
\end{defn}
By \cite[Prop.\ 5.3]{BaHeMo17} we have
\[(u|_{U})^{-1}\in W^{1,1}(\imT(u,U),\R^{n})\quad\text{and}\quad D (u|_{U})^{-1}=\left(Du\circ (u|_{U})^{-1}\right)^{-1}\text{ a.e.}\]

\subsection{Cut-and-paste and composition}\label{subse:results}

In this subsection we provide some auxiliary results for functions in $\A_p$.
To be precise, for the recovery sequence of Section \ref{sect:caprelaxrel} a typical construction is to cut and paste functions in $\Ap$, as well as to compose a Lipschitz function with one in $\Ap$.
The main aim of this subsection is to show that, under suitable assumptions, these two operations make a new function still in $\Ap$.

The following lemma shows that when we paste two functions in the class $\A_p$ that coincide in a neighborhood of a sphere, the resulting function is also in $\A_p$.
Note that it is not sufficient that the two functions coincide on the sphere, because a cavity may appear at a point of the sphere, and, hence, the resulting function will not be in $\Ap$ (this phenomenom is known as \emph{cavitation at the boundary}; see \cite{MuSp95,SiSp00,SiSpTi06,Henao09,HeRo17}).

\begin{lem}\label{lem paste Apq}
Let $p > n-1$ and $q \geq 1$.
Let $B, B'$ be open sets such that $B'\ssubset B\ssubset\Omega$.
Assume $u\in\A_{p,q}(\Omega)$, $v\in \A_{p,q}(B)$ and $u=v$ a.e.\ in $B\setminus B'$. Then the function
\[w := \begin{cases}
v & \text{in } B' ,\\
u &\text{in } \Omega\setminus B'
\end{cases}\]
is in $\A_{p,q}(\Omega)$.
If, in addition, $u\in\A_{p}^{1}(\Omega)$ and $v\in\A_{p}^{1}(B)$, then $w\in\A_{p}^{1}(\Omega)$.
\end{lem}
\begin{proof}
All the conditions in the definition of $\A_{p,q}$ are immediate to check except \eqref{eq:E0}, so let $\phi\in C^{1}_{c}(\Omega)$ and $g\in C^{1}_{c}(\R^{n},\R^{n})$.
Fix an $\eta \in C^1_c (\Om)$ with support in $B$ such that $\eta = 1$ in $B'$.
Then, as $\eta \, \phi \in C^1_c (B)$,
\begin{align*}
 \E_{\Om} (w,\phi,g) & = \E_{\Om} (w,\eta \, \phi,g) + \E_{\Om} (w,(1-\eta) \, \phi,g) = \E_{B} (w,\eta \, \phi,g) + \E_{\Om \setminus \bar{B}'} (w,(1-\eta) \, \phi,g) \\
 & = \E_{B} (v,\eta \, \phi,g) + \E_{\Om \setminus \bar{B}'} (u,(1-\eta) \, \phi,g) = 0 + \E_{\Om} (u,(1-\eta) \, \phi,g) = 0 + 0 = 0 .
\end{align*}
This concludes the proof also in the case $u\in\A_{p}^{1}(\Omega)$ and $v\in\A_{p}^{1}(B)$.
\end{proof}

In the next lemma we see that $\E_{\Omega}(u,\phi,g)$ is also zero for $u\in \A_{p,q}$ and $\phi$ in the correct Sobolev space.

\begin{lem}\label{lem energy for sobolev}
Let $p>n-1$ and $q> 1$.
Let $u\in \A_{p,q}(\Omega)$, $g\in C^{1}_{c}(\R^{n},\R^{n})$ and $\phi\in W^{1,q'}_0 (\Omega)\cap L^{\infty}(\Omega)$. Then $\E_{\Omega}(u,\phi,g)=0$.
\end{lem}
\begin{proof}
Let $\{ \phi_j \}_{j \in \N}$ be a sequence in $C^{1}_{c}(\Omega)$ such that $\phi_j \to \phi$ in $W^{1,q'}(\Omega)$ and $\phi_j \weakcs \phi$ in $L^{\infty} (\Om)$ as $j\to\infty$.
This sequence can be constructed as follows: first one takes a sequence $\{ \tilde{\phi}_j \}_{j \in \N}$ in $C^{1}_{c}(\Omega)$ such that $\tilde{\phi}_j \to \phi$ in $W^{1,q'}(\Omega)$ and a.e., and $D \tilde{\phi}_j \to D \phi$ a.e\@.
Then, one defines $\bar{\phi}_j = \max \{ \tilde{\phi}_j, \left\| \phi \right\|_{L^{\infty}} +1 \}$.
It is easy to check that $\bar{\phi}_j \to \phi$ in $W^{1,q'}(\Omega)$ and $\bar{\phi}_j \weakcs \phi$ in $L^{\infty} (\Om)$.
Then, one takes $\phi_j$ as a suitable mollification of $\bar{\phi}_j$.
When such $\phi_j$ have been constructed, we have $\E_{\Omega}(u,\phi_j,g)=0$ for all $j \in \N$, and
\begin{align*}
 & \lim_{j\to\infty} \int_{\Omega}\cof Du(x)\cdot (g(u(x))\otimes D\phi_j(x)) + \det D u(x) \, \phi_j(x)\Div g(u(x))\, \dd x \\
 & = \int_{\Omega}\cof Du(x)\cdot (g(u(x))\otimes D\phi(x)) + \det D u(x) \, \phi(x) \Div g(u(x)) \, \dd x ,
\end{align*}
so $\E_{\Omega}(u,\phi,g) = 0$.
\end{proof}

We prove that the composition of a function in the class $\A_{p,q}$ with a Lipschitz function satisfying some conditions is still in the class $\A_{p,q}$.
The assumptions may look artificial, but we will see in Section \ref{sect:caprelaxrel} that they will all be satisfied.

\begin{lem}\label{lem inner comp by lip}
Let $p>n-1$ and $q> 1$.
Let $u\in\A_{p,q}(\Omega)$, $B\ssubset\Omega$ a ball, $\rho: B\to\bar{B}$ Lipschitz such that $\rho|_{\partial B}=\id|_{\partial B}$, $\det D\rho>0$ a.e.\ and $\int_{\Omega}\left(\det D\rho\right)^{1-q'}\, \dd x<\infty$.
Define
\[z :=\begin{cases}
u \circ \rho & \text{in } B,\\
u &\text{in } \Omega\setminus B.
\end{cases}\]
Assume that $z\in W^{1,p}(\Omega, \R^n)$, $Dz= (Du \circ \rho) \, D\rho$ in $B$, $\det Dz\in L^{1} (\Omega)$ and $\cof Dz\in L^{q}(\Omega, \Rnn)$.
Then $z\in \A_{p,q}(\Omega)$.
If, in addition, $u\in\A_{p}^{1}(\Omega)$ and $\det D\rho=1$ a.e., then $z\in\A_{p}^{1}(\Omega)$.
\end{lem}
\begin{proof}
By definition of $\A_{p,q}$, to prove $z\in\A_{p,q}(\Omega)$ we only have to show that $\E_{\Omega}(z,\phi,g)=0$ for all $\phi\in C^{1}_{c}(\Omega)$ and $g\in C^{1}_{c}(\R^{n},\R^{n})$.
We have
\begin{equation*}
\E_{\Omega}(z,\phi,g)=\E_{\Omega\setminus \bar{B}}(u,\phi,g)+\E_{B}(u\circ \rho,\phi,g).
\end{equation*}
By \cite[Th.\ 8]{Sverak88} or \cite[Th.\ 3.3]{HeMo15}, we have $\rho^{-1}\in W^{1,1}(B, \R^n)$ and
\[
 D\rho^{-1} (y) = D \rho (\rho^{-1} (y))^{-1}, \ \ \det D\rho (\rho^{-1}(y))=\frac{1}{\det D\rho^{-1}(y)}, \ \ \cof D\rho(\rho^{-1}(y))=\frac{D\rho^{-1}(y)^{T}}{\det D\rho^{-1}(y)}
\]
for a.e.\ $y \in B$, so, by a change of variables,
\begin{equation}\label{eq:Eurho}
\begin{split}
\E_{B}(u\circ \rho,\phi,g)=&\int_{B}\left[\cof (Du(y)) \, D\rho^{-1}(y)^{T}\cdot (g(u(y))\otimes D\phi(\rho^{-1}(y)))\right.\\
&\left.+\det (D u(y)) \, \phi(\rho^{-1}(y))\Div g(u(y))\right] \dd y.
\end{split}
\end{equation}
By the chain rule (see, e.g., \cite[Th.\ 2.1.11]{Ziemer89}) we get that $\phi \circ \rho^{-1} \in W^{1,1}(B)$ and $D(\phi \circ \rho^{-1}) = (D\phi \circ \rho^{-1}) \, D\rho^{-1}$ in $B$.
In fact, $\phi \circ \rho^{-1} \in W^{1,q'}(B)$ since, changing variables and using the fact that $\rho$ and $\phi$ are Lipschitz, we get
\begin{align*}
 \left\| D (\phi \circ \rho^{-1}) \right\|_{L^{q'}(B)}^{q'} & \lesssim \left\|D \rho^{-1} \right\|_{L^{q'}(B)}^{q'} =\int_{B} \left| D\rho^{-1}(y) \right|^{q'}\, \dd y = \int_{B} \left| D\rho^{-1}(\rho(x)) \right|^{q'} \det D\rho (x) \, \dd x \\
&=\int_{B} \left| \cof D\rho(x) \right|^{q'} \det D\rho (x)^{1-q'} \, \dd x \lesssim \int_{B} \det D\rho (x)^{1-q'} \, \dd x<\infty .
\end{align*}
Equality $\E_{B} (u\circ \rho,\phi,g) = \E_{B} (u,\phi \circ \rho^{-1},g)$ is clear in view of \eqref{eq:Eurho}.
Define
\[
 \tilde{\phi} := \begin{cases}
 \phi & \text{in } \Om \setminus B , \\
 \phi \circ \rho^{-1} & \text{in } B .
 \end{cases}
\]
As $\rho|_{\partial B}=\id|_{\partial B}$, we have that $\tilde{\phi}$ is Sobolev; in fact, $\tilde{\phi} \in W^{1,q'}_0 (\Om)$.
Thanks to Lemma \ref{lem energy for sobolev} we have $\E_{\Omega}(u,\tilde{\phi},g)=0$, so
\[
 0 = \E_{\Omega} (u,\tilde{\phi},g) = \E_{\Omega \setminus \bar{B}} (u, \phi, g) + \E_{B} (u,\phi \circ \rho^{-1},g) = \E_{\Omega\setminus \bar{B}}(u,\phi,g)+\E_{B}(u\circ \rho,\phi,g) = \E_{\Omega}(z,\phi,g)
\]
and, hence, $z\in\A_{p,q}(\Omega)$.

If, in addition, $u\in\A_{p}^{1}(\Omega)$ and $\det D\rho=1$ a.e.\ then $\det Dz(x)=\det Du(\rho(x))\det D\rho(x)=1$ for a.e.\ $x\in B$ and $\det Dz(x)=\det Du(x)=1$ for a.e.\ $x\in \Omega\setminus B$. Therefore, $z\in\A_{p}^{1}(\Omega)$.
\end{proof}

\section{Compactness, lower semicontinuity and existence}\label{se:existence}

In this section we prove existence of minimizers of $I$ under the assumptions that $W$ is polyconvex in the first variable and $V$ is tangentially quasiconvex.

We first define the set of admissible functions.
We will distinguish two cases, according to whether the material is compressible (admissible set $\B$ and energy functional $I$) or incompressible (admissible set $\B_1$ and energy functional $I_1$).
The energy functional is, in principle, defined in the whole $L^1 (\Om, \Rn) \times L^1 (\Rn, \Rn)$ but it will be infinity outside the set of admissible functions.

Fix $p>n-1$ and $s > 1$.
Let $\Om \subset \Rn$ be a bounded Lipschitz domain.
Let $\Gamma$ be an $(n-1)$-rectifiable subset of $\partial \Om$, and let $u_0:\Gamma\to \R^{n}$.
We define $\B$ as the set of $(u,\vn) \in L^1 (\Om, \Rn) \times L^1 (\Rn, \Rn)$ such that $u\in \A_p (\Om)$, $u|_{\Gamma}=u_0$ in the sense of traces, $D u (x) \in \Rnn_+$ for a.e.\ $x \in \Om$,
\[
 \vn|_{\imT(u,\Omega)} \in W^{1,s} (\imT (u, \Om) ,\mS^{n-1}) \quad \text{and} \quad \vn|_{\R^n \setminus \imT(u,\Omega)} = 0 .
\]
Note that no boundary conditions are prescribed for $\vn$.
As for the incompressible case, we define $\B_1$ as the set of $(u,\vn) \in \B$ such that $D u (x) \in SL(n)$ for a.e.\ $x \in \Om$.

We define the energy functionals
\begin{equation}\label{eq:Idefinitions}
 I, \, I_{\mec}, \, I_{\nem}, \, I_1, \, I_{1,\mec}, \, I_{1,\nem} : L^1 (\Om, \Rn) \times L^1 (\Rn, \Rn) \to [0,\infty]
\end{equation}
describing the nematic elastomer as follows:
\[
 I_{\mec} (u, \vn) = \begin{cases}
 \displaystyle \int_{\Omega} W(Du(x),\vn(u(x)))\, \dd x , & \text{if } (u, \vn) \in \B , \\
 \infty , & \text{otherwise,}
 \end{cases}
\]
\[
 I_{\nem} (u, \vn) = \begin{cases}
 \displaystyle \int_{\imT(u,\Omega)} V(\vn(y), D\vn(y)) \, \dd y  , & \text{if } (u, \vn) \in \B , \\
 \infty , & \text{otherwise,}
 \end{cases}
\]
\[
 I_{1,\mec} (u, \vn) = \begin{cases}
 I_{\mec} (u, \vn) , & \text{if } (u, \vn) \in \B_1 , \\
 \infty , & \text{otherwise,}
 \end{cases}
 \qquad 
 I_{1,\nem} (u, \vn) = \begin{cases}
 I_{\nem} (u, \vn) , & \text{if } (u, \vn) \in \B_1 , \\
 \infty , & \text{otherwise.}
 \end{cases}
\]
Finally, $I:= I_{\mec} + I_{\nem}$ and $I_1 : = I_{1,\mec} + I_{1,\nem}$.

The following result establishes the lower semicontinuity of $I$ in $\B$ with respect to the $L^1$ topology.
Its proof is essentially a rewriting of the proofs of \cite[Props.\ 7.1, 7.8 and Th.\ 8.2]{BaHeMo17}, and will only be sketched.

\begin{prop}\label{prop:lowersemicont}
Let $s > 1$ and $p>n-1$.
Let
\begin{equation}\label{eq:ujnjun}
 (u_j, \vn_j) \to (u, \vn) \text{ in } L^1 (\Om, \R^n) \times L^1 (\Rn, \Rn) \qquad \text{as } j\to\infty .
\end{equation}
Let $W:\R^{n\times n}_{+}\times \mS^{n-1}\to [0,\infty)$ be continuous, polyconvex and such that
\begin{equation}\label{eq:coercWtilde}
 W (F, \vn) \geq c \, |F|^p + \theta (\det F) , \qquad F \in \Rnn_+ , \quad \vn \in \mS^{n-1}
\end{equation}
for a constant $c>0$ and a Borel function $\theta : (0, \infty) \to [0, \infty)$ with
\begin{equation}\label{eq:h1}
 \lim_{t \searrow 0} \theta (t) = \lim_{t \to \infty} \frac{\theta (t)}{t} = \infty.
\end{equation}
Let $V : T^n \mS^{n-1} \to [0, \infty)$ be continuous and tangentially quasiconvex such that
\begin{equation}\label{eq:Vsgrowth}
 c \left| \zeta \right|^s - \frac{1}{c} \leq V (z, \zeta) \leq \frac{1}{c} \left( 1 + \left| \zeta \right|^s \right) ,Ê\qquad (z, \zeta) \in T^n \mS^{n-1} .
\end{equation}
Then
\begin{equation}\label{eq:Isemicont}
 I (u, \vn) \leq \liminf_{j \to \infty} I (u_j, \vn_j) .
\end{equation}
\end{prop}
\begin{proof}
By taking a subsequence, we can assume that the $\liminf$ of the right-hand side of \eqref{eq:Isemicont} is a limit, and that, in fact, it is finite.
The proof of \cite[Th.\ 8.2]{BaHeMo17} shows that
\begin{align*}
 & u_j \weakc u \text{ in } W^{1,p} (\Om, \Rn) , \qquad \det D u_j \weakc \det D u \text{ in } L^1 (\Om) , \\  & \chi_{\imT (u_j, \Om)} D \vn_j \weakc \chi_{\imT (u, \Om)} D \vn \quad\text{ in } L^s (\Rn, \Rnn) \qquad \text{as } j \to \infty
\end{align*}
where $\chi_{\imT (u, \Om)} D \vn$ stands for the extension of $D \vn$ by zero outside $\imT (u, \Om)$, and analogously for $\chi_{\imT (u_j, \Om)} D \vn_j$,
and, by \cite[Prop.\ 7.8]{BaHeMo17},
\[
 I_{\mec} (u, \vn) \leq \liminf_{j \to \infty} I_{\mec} (u_j, \vn_j) .
\]

Now let $G \ssubset \imT (u, \Om)$ be open.
Then, by \cite[Lemma 3.6]{BaHeMo17}, there exists $j_0 \in \N$ such that for all $j \geq j_0$ we have $G \subset \imT (u_j, \Om)$.
Therefore, $\vn_j \weakc \vn$ in $W^{1,s} (G, \Rn)$ as $j \to \infty$, so by Theorem \ref{th:tqc},
\[
 \int_{G} V (\vn (y), D \vn (y)) \, \dd y \leq \liminf_{j \to \infty} \int_{G} V (\vn_j (y), D \vn_j (y)) \, \dd y
 \leq \liminf_{j \to \infty } I_{\nem} (u_j, \vn_j) .
\]
As this is true for all open $G \ssubset \imT (u, \Om)$ we obtain
\[
 I_{\nem} (u, \vn) \leq \liminf_{j \to \infty} I_{\nem} (u_j, \vn_j) ,
\]
which concludes the proof.
\end{proof}

The compactness for sequences bounded in energy is as follows.
Its proof, again, is a rewriting of that of \cite[Prop.\ 7.1 and Th.\ 8.2]{BaHeMo17} and will be omitted.

\begin{prop}\label{prop:compactness}
Let $s > 1$ and $p>n-1$.
Let $W:\R^{n\times n}_{+}\times \mS^{n-1}\to [0,\infty)$ satisfy \eqref{eq:coercWtilde}--\eqref{eq:h1}
for a constant $c>0$ and a Borel function $\theta : (0, \infty) \to [0, \infty)$.
Let $V : T^n \mS^{n-1} \to [0,\infty)$ satisfy
\begin{equation}\label{eq:Vcoerc}
 c \left| \zeta \right|^s - \frac{1}{c} \leq V (z, \zeta) ,Ê\qquad (z, \zeta) \in T^n \mS^{n-1} .
\end{equation}
For each $j\in \N$, let $(u_j, \vn_j) \in L^1 (\Om, \Rn) \times L^1 (\Rn, \Rn)$ satisfy
\[
 \sup_{j \in \N} I(u_j, \vn_j) < \infty .
\]
Then there exist a subsequence (not relabelled) and $(u, \vn) \in \B$ such that \eqref{eq:ujnjun} holds.
\end{prop}

Proposition \ref{prop:lowersemicont} and \ref{prop:compactness} yield, by the direct method of the calculus of variations, the following result on the existence of minimizers.

\begin{theorem}\label{th:existence}
Let $s > 1$ and $p>n-1$.
Let $W:\R^{n\times n}_{+}\times \mS^{n-1}\to [0,\infty)$ be continuous, polyconvex and such that \eqref{eq:coercWtilde}--\eqref{eq:h1} hold for a constant $c>0$ and a Borel function $\theta : (0, \infty) \to [0, \infty)$.
Let $V : T^n \mS^{n-1} \to [0, \infty)$ be continuous and tangentially quasiconvex such that \eqref{eq:Vsgrowth} holds.
If $\B\neq \varnothing$ and $I$ is not identically infinity, then $I$ attains its minimum in $\B$.
\end{theorem}

In the incompressible case, the analogue results are as follows; as commented in \cite[Rks.\ 7.9 and 8.4]{BaHeMo17}, the incompressibility can easily be taken into account.

\begin{prop}\label{prop:lowersemicontIncomp}
Let $s > 1$ and $p>n-1$.
Let \eqref{eq:ujnjun} hold.
Let $W: SL(n) \times \mS^{n-1}\to [0,\infty)$ be continuous, polyconvex and such that
\begin{equation}\label{eq:coercWtildeIncomp}
 W (F, \vn) \geq c \, |F|^p - \frac{1}{c} , \qquad F \in SL(n) , \quad \vn \in \mS^{n-1}
\end{equation}
for a constant $c>0$.
Let $V : T^n \mS^{n-1} \to [0, \infty)$ be continuous and tangentially quasiconvex such that bound \eqref{eq:Vsgrowth} holds.
Then
\[
 I_1 (u, \vn) \leq \liminf_{j \to \infty} I_1 (u_j, \vn_j) .
\]
\end{prop}

\begin{prop}\label{prop:compactnessIncomp}
Let $s > 1$ and $p>n-1$.
Let $W: SL(n) \times \mS^{n-1}\to [0,\infty)$ satisfy \eqref{eq:coercWtildeIncomp} for a constant $c>0$.
Let $V : T^n \mS^{n-1} \to [0,\infty)$ satisfy \eqref{eq:Vcoerc}.
For each $j\in \N$, let $(u_j, \vn_j) \in L^1 (\Om, \Rn) \times L^1 (\Rn, \Rn)$ satisfy
\[
 \sup_{j \in \N} I_1 (u_j, \vn_j) < \infty .
\]
Then there exist a subsequence (not relabelled) and $(u, \vn) \in \B_1$ such that \eqref{eq:ujnjun} holds.
\end{prop}

\begin{theorem}\label{th:existenceIncomp}
Let $s > 1$ and $p>n-1$.
Let $W: SL(n) \times \mS^{n-1}\to [0,\infty)$ be continuous, polyconvex and such that \eqref{eq:coercWtildeIncomp} holds for a constant $c>0$.
Let $V : T^n \mS^{n-1} \to [0, \infty)$ be continuous and tangentially quasiconvex such that \eqref{eq:Vsgrowth} holds.
If $\B_1 \neq \varnothing$ and $I_1$ is not identically infinity, then $I_1$ attains its minimum in $\B_1$.
\end{theorem}

\section{Product and chain rule for Sobolev functions}\label{se:productchain}

In this section we state three results proved in \cite{CoDo15} about the product of $L^1$ functions and the composition of a Lipschitz function with a Sobolev function.

The next lemma, taken from \cite[Lemma 3.1]{CoDo15}, states that there are many translations such that the product of the translated $L^{1}$ functions is in $L^{1}$.
\begin{lem}\label{lem composition L1}
Let $x_0\in\R^n$ and $r>0$.
Let $\psi\in W^{1,\infty}(B(0,r),\bar{B}(0,r))$, $g\in L^{1}(B(0,r))$ and $f\in L^{1}(B(x_0,2r))$.
Then, there exists a measurable set $E\subset B(x_0,r)$ of positive measure such that for any $a_0\in E$, the function
\[\tilde{f} (x) :=f(a_0 + \psi(x-a_0)) \, g(x-a_0) , \qquad x \in B(a_0,r) \]
belongs to $L^{1}(B(a_0,r))$ and 
\[\|\tilde{f}\|_{L^{1}(B(a_0,r))}\le \frac{1}{|B(0,r)|}\|f\|_{L^{1}(B(x_0,2r))}\|g\|_{L^{1}(B(0,r))}.\]
\end{lem}

The following result is a weaker version of \cite[Lemma A.1]{CoDo15}.
\begin{lem}\label{lem:L1lipisL1}
Let $\psi\in W^{1,\infty}(B(0,1),\bar{B}(0,1))$ and $f\in L^{1}(B(0,2))$. Then the map $(x,a_0) \mapsto f(a_0+\psi(x-a_0))$ is measurable and for almost all $a_0\in B(0,1)$ the function
\[x \mapsto f(a_0+\psi(x-a_0))\]
is in $L^{1}(B(a_0,1))$.
\end{lem}

The following version of the chain rule was proved in \cite[Lemma A.2]{CoDo15}.
\begin{lem}\label{lem:chainrule}
Let $\psi\in W^{1,\infty}(B(0,1),\bar{B}(0,1))$ and $u\in W^{1,1}(B(0,2))$. Then for almost all $a_0\in B(0,1)$ the function
\[
 w(x) := u(a_0+\psi(x-a_0)) , \qquad x \in B(a_0,1)
\]
belongs to $W^{1,1}(B(a_0,1))$ and
\[Dw(x)=Du(a_0+\psi(x-a_0)) \, D\psi(x-a_0).\]
If, in addition, $\psi = \id$ on $\partial B(0,1)$ then $w=u$ on $\partial B(a_0,1)$ in the sense of traces.
\end{lem}

\section{Recovery sequence}\label{sect:caprelaxrel}

In this section we prove the upper bound inequality by constructing a recovery sequence.

We first present the coercivity, growth and continuity conditions of the energy functions $W$ and $V$, which are slightly more restrictive that those of Section \ref{se:existence}.
Fix $p>n-1$, $q>1$ and $s > 1$.
In the compressible case, the conditions on $W$ are as follows.

\begin{enumerate}[(W)]

\item $W: \R^{n\times n}_{+} \times \mS^{n-1} \to [0,\infty)$ is continuous and there exist a convex $\theta:(0,\infty)\to [0,\infty)$, a bounded Borel $h:[0,2] \to [0,\infty)$ and $c>0$ such that 
\begin{align*}
 & \theta(t_1 t_2) \lesssim \left( 1 + \theta(t_1) \right) \left( 1+\theta(t_2) \right) , \qquad t_1, t_2 > 0 , \\
 & \lim_{t\to\infty}\frac{\theta(t)}{t}=\infty , \qquad \liminf_{t\to 0}t^{q'-1}\theta(t)>0, \qquad \lim_{t\to 0} h(t)=0 ,
 \end{align*}
 and for all $F \in \Rnn_+$ and $\vn,\vm \in \mS^{n-1}$,
\begin{align*}
 & \frac{1}{c} \left( \left| F \right|^{p} + \left| \cof F \right|^q  + \theta(\det F) \right) -c \leq W(F, \vn) \leq c \left( \left| F \right|^{p} +  \theta(\det F) + 1 \right) , \\
 & \left| W(F,\vn)-W(F,\vm) \right|\le h \left( |\vn-\vm| \right) W(F,\vn)  .
\end{align*}
\end{enumerate}
The function $W$ is extended to $(\Rnn \setminus \R^{n\times n}_{+}) \times \mS^{n-1}$ by infinity.
Observe that if $(u, \vn) \in \B$ satisfies $I_{\mec} (u, \vn) < \infty$ then $u \in \A_{p,q} (\Om)$.

In the incompressible case, the conditions on $W$ are:

\begin{enumerate}[(W$_1$)]

\item $W: SL(n) \times \mS^{n-1} \to [0,\infty)$ is continuous and there exist a bounded Borel $h:[0,2] \to [0,\infty)$ and $c>0$ such that $\lim_{t\to 0} h(t)=0$ and for all $F \in SL(n)$ and $\vn,\vm \in \mS^{n-1}$,
\begin{align*}
 & \frac{1}{c} |F|^{p}  -c \leq W(F, \vn) \leq c \left( |F|^{p} + 1 \right) , \\
 & \left| W(F,\vn)-W(F,\vm) \right|\le h \left( |\vn-\vm| \right) W(F,\vn)  .
\end{align*}
\end{enumerate}
The function $W$ is extended to $(\Rnn \setminus SL(n)) \times \mS^{n-1}$ by infinity.
Note that $q$ does not play any role in the incompressible case.

The assumption for $V$ is as follows:
\begin{enumerate}[(V)]
\item $V : T^n \mS^{n-1} \to [0, \infty)$ is continuous and there exists $c>0$ such that
\[
\frac{1}{c}|\zeta|^{s}-c\leq V(z, \zeta)\leq c|\zeta|^{s}+c , \qquad (z, \zeta) \in T^n \mS^{n-1}.
\]
\end{enumerate}

We define the admissible spaces $\B$ and $\B_1$ as in Section \ref{se:existence}, as well as the functionals \eqref{eq:Idefinitions}.
We also define the functionals
\[
 I^{*}, \, I^{*}_{\nem} , \, I^{*}_{\mec}, \, I_1^{*}, \, I^{*}_{1,\nem} , \, I^{*}_{1,\mec} : L^1 (\Om, \Rn) \times L^1 (\Rn, \Rn) \to [0,\infty]
\]
in a similar way to their counterparts \eqref{eq:Idefinitions}, but replacing $W$ with $W^{qc}$ and $V$ with $V^{tqc}$, i.e.,
\[
 I^*_{\mec} (u, \vn) = \begin{cases}
 \displaystyle \int_{\Omega} W^{qc} (Du(x),\vn(u(x)))\, \dd x , & \text{if } (u, \vn) \in \B , \\
 \infty , & \text{otherwise,}
 \end{cases}
\]
\[
 I^*_{\nem} (u, \vn) = \begin{cases}
 \displaystyle \int_{\imT(u,\Omega)} V^{tqc} (\vn(y), D\vn(y)) \, \dd y  , & \text{if } (u, \vn) \in \B , \\
 \infty , & \text{otherwise,}
 \end{cases}
\]
\[
 I^*_{1,\mec} (u, \vn) = \begin{cases}
 I^*_{\mec} (u, \vn) , & \text{if } (u, \vn) \in \B_1 , \\
 \infty , & \text{otherwise,}
 \end{cases}
 \qquad 
 I^*_{1,\nem} (u, \vn) = \begin{cases}
 I^*_{\nem} (u, \vn) , & \text{if } (u, \vn) \in \B_1 , \\
 \infty , & \text{otherwise,}
 \end{cases}
\]
$I^*:= I^*_{\mec} + I^*_{\nem}$ and $I^*_1 : = I^*_{1,\mec} + I^*_{1,\nem}$.
Here $W^{qc}$ is the quasiconvexification of $W$ with respect to the first variable.

Now we state the main result of this section: the existence of the recovery sequence.
It follows from Lemma \ref{lem conv uj} below.

\begin{theorem}\label{th conv nem det>0}
Let $q> 1$, $s> 1$, $p>n-1$, $V$ satisfy (V) and $W$ satisfy (W) (respectively, (W$_1$)). Let $\Omega \subset \R^{n}$ open bounded and Lipschitz. Then, for any $(u, \vn) \in L^1 (\Om, \Rn) \times L^1 (\Rn, \Rn)$ there is a sequence $\{(u_j, \vn_j)\}_{j\in\N}\subset L^1 (\Om, \Rn) \times L^1 (\Rn, \Rn)$ such that
\[
 (u_j, \vn_j) \to (u, \vn) \text{ in } L^1 (\Om, \R^n) \times L^1 (\Rn, \Rn) \qquad \text{as } j\to\infty
\]
and
\[
 \limsup_{j\to\infty} I (u_j, \vn_j) \le I^* (u, \vn) \qquad \text{(respectively, }\limsup_{j\to\infty} I_1 (u_j, \vn_j) \le I^*_1 (u, \vn) \text {).}
\]
\end{theorem}

The proof is divided into two lemmas.
In the first one, the function $u$ is modified in a ball.

\begin{lem}\label{lem approx u by z}
Assume one of the following:
\begin{enumerate}[a)]
\item\label{appr cond det>0}  $W$ satisfies (W),
\item\label{appr cond det=1} $W$ satisfies (W$_1$),
\end{enumerate}
and fix $F\in\R_{+}^{n\times n}$ in case \emph{\ref{appr cond det>0})}  and  $F\in SL(n)$ in case \emph{\ref{appr cond det=1})}, $\vm\in\mS^{n-1}$ and $\eta\in (0,1)$. Then there is $\delta>0$ such that for any ball $B=B(x_0,r)$, any $\vn\in L^{\infty} (\imT(u,B),\mS^{n-1})$ and any
 \[u\in\begin{cases}
 \A_{p,q}(B)&\text{in case \emph{\ref{appr cond det>0})},} \\
  \A_{p}^{1}(B)&\text{in case \emph{\ref{appr cond det=1})}} 
 \end{cases}\]
  with
\begin{equation}\label{eq: approxledelta}
\fint_B \left(|Du-F|^{p}+|\theta(\det Du)-\theta(\det F)|+|\vn\circ u-\vm|^{p}\right) \dd x\leq \delta\text{ in case \emph{\ref{appr cond det>0})},}
\end{equation}
or
\begin{equation*}
\fint_B \left(|Du-F|^{p}+|\vn\circ u-\vm|^{p}\right) \dd x\leq \delta\text{ in case \emph{\ref{appr cond det=1})},}
\end{equation*}
there exist $a_0\in B\left(x_0,\frac{r}{2}\right)$ and
 \[z\in\begin{cases}
 \A_{p,q}(B)& \text{in case \emph{\ref{appr cond det>0})},} \\
  \A_{p}^{1}(B)& \text{in case \emph{\ref{appr cond det=1})}} 
 \end{cases}\]
 with $z=u$ in $B(x_0,r)\setminus B\left(a_0,\frac{r}{2}\right)$, $\imT (z, \Om) = \imT (u, \Om)$,
\begin{equation}\label{eq: lemconv Wqc}
\int_{B\left(a_0,\frac{r}{2}\right)}W(Dz, \vn\circ z)\, \dd x\leq \int_{B\left(a_0,\frac{r}{2}\right)}(W^{qc}(Du, \vn\circ u)+\eta)\, \dd x
\end{equation}
and
\begin{equation}\label{eq: lemconvlpbound}
\int_B|u-z|^{p} \, \dd x \leq c \, r^{p}\int_B(W^{qc}(Du,\vn\circ u)+1)\, \dd x
\end{equation}
for some $c>0$ depending on $W$, $n$ and $p$.
If $u$ is Lipschitz, then so is $z$. 
\end{lem}
\begin{proof}
This proof is partially based on that of \cite[Lemma 3.2]{CoDo15}.  We will only prove the case \emph{\ref{appr cond det>0})}, since the proof of case \emph{\ref{appr cond det=1})} is analogous.

The $L^{p}$ bound \eqref{eq: lemconvlpbound} follows from Poincar\'{e}'s inequality, the growth condition of (W) and \eqref{eq: lemconv Wqc} as follows:
\begin{align*}
 & \int_B|u-z|^p \, \dd x = \int_{B (a_0, \frac{r}{2})} |u-z|^p \, \dd x \lesssim r^p \int_{B (a_0, \frac{r}{2})} |Du - Dz|^p \, \dd x \\
 & \lesssim r^p \int_{B (a_0, \frac{r}{2})} \left( |Du|^p + |Dz|^p \right) \dd x \lesssim r^p \int_{B (a_0, \frac{r}{2})} \left( W^{qc} (Du, \vn \circ u) + W (Dz, \vn \circ z) + 1 \right) \dd x \\
 & \lesssim r^p \int_{B (a_0, \frac{r}{2})} \left( W^{qc} (Du, \vn \circ u) + 1 \right) \dd x \leq r^p \int_{B} \left( W^{qc} (Du, \vn \circ u) + 1 \right) \dd x ,
\end{align*}
so the bulk of the proof consists in showing \eqref{eq: lemconv Wqc}.

By Definition \ref{de:quasiconvexification} of quasiconvexification, there exists $\vp_{\eta}\in W^{1,\infty} (B (0,\frac{r}{2} ), \R^n )$ such that $\vp_{\eta}(x)=Fx$ on $\partial B (0,\frac{r}{2})$, $\det D\vp_{\eta}>0$ a.e.\ and
\begin{equation}\label{eq: approxqcvp}
\fint_{B(0, \frac{r}{2})} W(D\vp_{\eta},\vm)\, \dd x\leq W^{qc}(F,\vm)+\eta.
\end{equation}

The function $F^{-1}\vp_\eta$ is Lipschitz and is the identity on $\partial B (0,\frac{r}{2})$, hence, by degree theory (see, if necessary, \cite[Th.\ 1]{Ball81}), $F^{-1}\vp_\eta (B (0,\frac{r}{2} ) )\subset \bar{B} (0,\frac{r}{2} )$.
 Moreover, $F^{-1}\vp_\eta$ is invertible and its inverse is in $W^{1,1}$ (see \cite[Th.\ 8]{Sverak88} or \cite[Th.\ 3.3]{HeMo15}).
  Take $a_0\in B (x_0,\frac{r}{2} )$ (to be chosen below), call $B'=B (a_0,\frac{r}{2})$ and set $v(x)=F^{-1}\vp_\eta(x-a_0)+a_0$ and 
\[z=\begin{cases} u \circ v &\text{in } B' , \\
u&\text{in } B(x_0,r)\setminus B' .
\end{cases}\]
It is clear that $z=u$ in $B(x_0,r)\setminus B'$, $\imT(v,B')=B'$ and $v^{-1}\in W^{1,1}(B',\R^{n})$. 

By Lemmas \ref{lem:chainrule} and \ref{lem:L1lipisL1}, there exists a null set $N$ such that for all $a_0\in B(x_0, \frac{r}{2}) \setminus N$ we have that $z\in W^{1,1}(B',\R^{n})$, $\det Dz\in L^{1}(B)$ and $\cof Dz\in L^{q} (B, \Rnn)$. Moreover, since $v|_{\partial B'}=\id|_{\partial B'}$  we have $u\circ v|_{\partial B'}=u|_{\partial B'}$ and, hence, $z\in W^{1,1}(B,\R^{n})$. Choose $E$ and $a_0\in E\setminus N$ using Lemma \ref{lem composition L1} applied to $B'$ with $\psi=F^{-1}\vp_\eta$, $f=|Du-F|^{p}+|\theta(\det Du)-\theta(\det F)|$ and $g=1+\theta(\det(F^{-1}D\vp_\eta))$. Then, by \eqref{eq: approxledelta},
\begin{equation}\label{eq: approxqcledelta}
 \fint_{B'}(1+\theta(\det Dv))\left(|Du-F|^{p}+|\theta(\det Du)-\theta(\det F)|\right)\circ v\, \dd x\le c_\eta \, \delta,
\end{equation}
with $c_\eta$ depending on $\eta$ and $F$.

By (W) and \eqref{eq: approxqcvp} we have $\theta(\det D\vp_\eta)\in L^{1} (B (0,\frac{r}{2} ) )$, $\theta(\det (F^{-1} D\vp_\eta))\in L^{1} (B(0,\frac{r}{2}))$ and $(\det D\vp_{\eta})^{1-q'} \in L^1 (B(0,\frac{r}{2}))$. Therefore, there exists $\gamma>0$ (depending on $F$, $\vm$ and $\eta$) such that  
\begin{equation}\label{eq: approxqcdet<gam1}
\int_{B\left(0,\frac{r}{2}\right)\cap \{\det D\vp_\eta<\gamma\}}(1+\theta(\det(F^{-1}D\vp_\eta)))\, \dd x\le \frac{\left|B\left(0,\frac{r}{2}\right)\right|\eta}{\left(3+\|F^{-1}D\vp_\eta\|_{L^{\infty}}^{p}\right)\left(1+|F|^{p}+\theta(\det F)\right)}
\end{equation}
and
\begin{equation}\label{eq: approxqcdet<gam2}
\int_{B\left(0,\frac{r}{2}\right)\cap \{\det D\vp_\eta<\gamma\}}(1+ |D\vp_\eta|^{p}+\theta(\det D\vp_\eta))\, \dd x\le \frac{1}{c}\left|B\left(0,\frac{r}{2}\right)\right|\eta,
\end{equation}
where $c$ is the constant of (W).

Let  $R_\eta=\|Dv\|_{L^{\infty}}$ and $M_\eta=\|D\vp_\eta\|_{L^{\infty}}$.
 Since $W$ is continuous in $\R^{n\times n}_{+}$ there is $\ve>0$ not depending on $u$, $\vn$ or $\delta$ with $\ve R_\eta \le 1$ and $\ve \le 1$ such that
\begin{equation}\label{eq: approxconvqcWcont}
|W(\sigma,\vec{\ell})-W(\zeta,\vec{k})|\le \eta
\end{equation}
for all $\sigma,\zeta \in \R^{n\times n}_{+}$ and $\vl, \vec{k} \in \mS^{n-1}$ with $|\zeta|\le M_\eta$, $\det \zeta\ge\gamma$ and $|\sigma-\zeta|+|\vec{\ell}-\vec{k}|\le \ve R_{\eta}$.
Moreover, by Proposition \ref{pr:Wqccont} and the continuity of $\theta$, the number $\ve$ can be chosen so that
\begin{equation}\label{eq: approxconvqcWqccont}
|W^{qc}(\zeta,\vec{\ell})-W^{qc}(F,\vm)|+|\theta(\det \zeta)-\theta(\det F)|\le \eta
\end{equation}
for all $\zeta \in \Rnn_+$ and $\vec{\ell} \in \mS^{n-1}$ satisfying $|\zeta-F|+|\vm-\vec{\ell}|\le \ve$.


 Set $\hat{\vp}_\eta(x)=\vp_\eta(x-a_0)$, and write
\[\int_{B'}\left(W(Dz,\vn\circ z)-W^{qc}(Du,\vn\circ u)\right) \dd x=I_1+I_2+I_3+I_4,\]
with
\begin{align*}
 & I_1=\int_{B'}\left(W(Dz,\vn\circ z)-W(D\hvpe,\vn\circ z)\right) \dd x, \qquad
I_2=\int_{B'}\left(W(D\hvpe,\vn\circ z)-W(D\hvpe,\vm)\right) \dd x, \\
 & I_3=\int_{B'}\left(W(D\hvpe,\vm)-W^{qc}(F,\vm)\right) \dd x \quad \text{and} \quad 
I_4=\int_{B'}\left(W^{qc}(F,\vm)-W^{qc}(Du,\vn\circ u)\right) \dd x.
\end{align*}
We will estimate these four integrals separately.
Thanks to \eqref{eq: approxqcvp} we have $I_3\le \eta|B'|$. To estimate $I_4$ we use \eqref{eq: approxconvqcWqccont} to get 
\[W^{qc}(F,\vm)\le W^{qc}(Du,\vn\circ u)+\eta\quad\text{on the set where }|Du-F|+|\vm-\vn\circ u|\le \ve.\]
In $ \{x\in B': |Du(x)-F|+|\vm-\vn\circ u(x)|> \ve\}$ we use \eqref{eq: approxledelta} and Chebyshev's inequality to get
\begin{align*}
I_4&\le \eta \, | B'| + W^{qc}(F,\vm) \left| \{x\in B': |Du(x)-F|+|\vm-\vn\circ u(x)|> \ve\} \right|\\
&\le \eta \, |B'| + W^{qc}(F,\vm)\frac{2^{p-1}}{\ve^{p}} \, |B| \, \delta.
\end{align*}
To estimate $I_2$ we need to define the sets 
\[\omega=\{x\in B': |\vn\circ u(x)-\vm| \ge \ve R_\eta \} \quad \text{and} \quad \omega_d=\{x\in B': \det D\hvpe (x)\ge \gamma\},\]
 where $\ve$ and $\gamma$ are those of \eqref{eq: approxconvqcWcont}.
  Doing the change of variables $z(x)=u(x')$, i.e., $x=v^{-1}(x')$, we obtain
\begin{align*}
I_2=\int_{B'}\left(W((D\hvpe)\circ v^{-1}(x'),\vn\circ u(x'))-W((D\hvpe)\circ v^{-1}(x'),\vm)\right)\det D v^{-1}(x')\, \dd x',
\end{align*}
and
\[\int_{B'}\det D v^{-1}(x')\, \dd x'= |B'|.\]
Using (W)  and \eqref{eq: approxconvqcWcont} we get
\begin{align*}
I_2&\le \int_{v(\omega_d)\setminus\omega}\eta \det D v^{-1}(x')\, \dd x'+\int_{B'\setminus  (v(\omega_d)\setminus \omega)}W((D\hvpe)\circ v^{-1}(x'),\vn\circ u(x'))\det D v^{-1}(x')\, \dd x'\\
&\le \eta \, |B'|+c\int_{B'\setminus (v(\omega_d)\setminus\omega)}\left(1+ |(D\hvpe)\circ v^{-1}(x')|^{p}+\theta(\det D\hvpe)\circ v^{-1}(x')\right)\det D v^{-1}(x')\, \dd x'.
\end{align*}
Doing the change of variables $x=v^{-1}(x')$ and using \eqref{eq: approxqcdet<gam2} we obtain
\begin{align*}
&c\int_{B'\setminus v(\omega_d)}\left(1+ |(D\hvpe)\circ v^{-1}(x')|^{p}+\theta(\det D\hvpe)\circ v^{-1}(x')\right)\det D v^{-1}(x')\, \dd x'\\
&\le c\int_{B'\setminus \omega_d}\left(1+ |D\hvpe(x)|^{p}+\theta(\det D\hvpe(x))\right) \dd x\le \eta \, |B'|.
\end{align*}
 On the other hand, for $x\in \omega_d$ we have that $\det Dv(x)\ge \gamma\det F^{-1}$, so $\det D v^{-1}\in L^{\infty}(v(\omega_d))$. Then, using \eqref{eq: approxledelta}, $\theta(\det D\hvpe)\in L^{\infty}(\omega_d)$ and Chebyshev's inequality we get
\begin{align*}
&c\int_{\omega\cap v(\omega_d)}\left(1+ |(D\hvpe)\circ v^{-1}(x')|^{p}+\theta(\det (D\hvpe))\circ v^{-1}(x')\right)\det D v^{-1}(x')\, \dd x'\\
&\lesssim |\omega|\lesssim \ve^{-p} \, \delta \, |B'|,
\end{align*} 
with the constant under $\lesssim$ depends on $W$, $\gamma$ and $\eta$ but not on $\delta$, $\vn$, $u$ or $z$.

 Hence, we have that there exists a constant $\tilde{c}$ depending on $\eta$ and $W$ but not on $\delta$ such that
\begin{equation*}
I_2\le (2\eta+\tilde{c}\ve^{-p}\delta) \, |B'|.
\end{equation*}
Next, we estimate $I_1$. Let
\[\omega'=\{x\in B': |Du(x)-F|\circ v\ge \ve \}.\]
 Using that, in $B'$,
\[Dz=(Du\circ v)Dv=\left[(Du-F)\circ v\right]Dv+D\hvpe\] 
and that in $\omega_d\setminus \omega'$ we have $\det D\hvpe\ge \gamma$ and $ |Du(x)-F|\circ v \le \ve  $ we get
\[|Dz-D\hvpe|\le \left[|Du-F|\circ v\right]|Dv|\le \ve R_{\eta} .\]
By \eqref{eq: approxconvqcWcont} we have
\[\int_{\omega_d\setminus \omega'}\left(W(Dz,\vn\circ z)-W(D\hvpe,\vn\circ z)\right) \dd x\le \eta|B'|.\]
Using the growth estimate (W) we obtain
\[W(Dz,\vn\circ z)\le c\left(1+ \left[|Du|^{p}\circ v\right]|Dv|^{p}+\theta((\det Du)\circ v\det Dv)\right).\]
Hence using $|Dv|\le R_\eta$ and (W) we get that, in $B'$,
\begin{equation}\label{eq: approxboundWdz}
W(Dz,\vn\circ z)\le c\left(1+R_\eta^{p} |Du|^{p}\circ v+1+\theta((\det Du)\circ v)\right)(1+\theta (\det Dv)).
\end{equation}
To estimate the integral in $\omega'$ we observe that $|Du-F|\circ v\ge \ve $ implies
\[|Du|\circ v+1\le |Du-F|\circ v+|F|+1\le\left(\frac{|F|+1}{\ve}+1\right)|Du-F|\circ v\]
and
\[\theta (\det Du)\circ v\le |\theta(\det Du)\circ v-\theta(\det F)|+\frac{\theta (\det F)}{\ve^{p}} \, |Du-F|^{p}\circ v.\]
Therefore, from \eqref{eq: approxboundWdz} and \eqref{eq: approxqcledelta} we obtain
\begin{align*}
\int_{\omega'}W(Dz,\vn\circ z)&\le c\int_{\omega'}(1+\theta (\det Dv(x)))\left(2+R_\eta^{p} |Du|^{p}+\theta(\det Du)\right)\circ v(x)\, \dd x\\
&\le c'\int_{\omega'}(1+\theta (\det Dv(x)))\left(|Du-F|^{p}+|\theta(\det Du)-\theta(\det F)|\right)\circ v(x)\, \dd x\\
&\le c_\eta' \, \delta \, |B'|. 
\end{align*}
The constant $c_\eta'$ depends on $W$, $\eta$ and $F$ but not on $\delta$.
 In $B'\setminus (\omega_d\cup\omega')$ we have $|Du-F|\circ v\le \ve  \le 1$ and $\det D\hvpe < \gamma$.
  Then we have $|Du|\circ v\le |F|+1$ and thanks to \eqref{eq: approxconvqcWqccont} we also obtain $\theta(\det Du)\circ v\le \theta(\det F)+1$.
   Therefore \eqref{eq: approxboundWdz} implies
\begin{align*}
W(Dz,\vn\circ z)&\le c\left(3+R_\eta^{p} (1+|F|)^{p}+\theta(\det F)\right)(1+\theta (\det Dv))\\
& \le c_*\left(3+\|F^{-1}D\vp_\eta\|_{L^{\infty}}^{p}\right)\left(1+|F|^{p}+\theta(\det F)\right)(1+\theta (\det Dv)),
\end{align*}
with $c_*$ depending only on $W$. Hence, thanks to \eqref{eq: approxqcdet<gam1} we get
\[\int_{B'\setminus (\omega'\cup\omega_d)}W(Dz,\vn\circ z)\, \dd x\le c_* \, \eta \, |B'|.\]
Consequently,
\[I_1\le (\eta+c_\eta'\delta+c_*\eta) \, |B'|.\]
Adding the estimates for  $I_1$, $I_2$, $I_3$ and $I_4$ we obtain
\begin{align*}
\int_{B'}&\left(W(Dz,\vn\circ z)-W^{qc}(Du,\vn\circ u)\right) \dd x\\
&\le\left(\eta+c_\eta'\delta+c_*\eta+2\eta+\frac{\tilde{c}}{\ve^{p}}\delta+\eta+\eta+W^{qc}(F,\vm)\frac{2^{p-1}}{\ve^{p}}\delta\right)|B'|.
\end{align*}
Recall that $\eta$, $c_\eta'$, $c_*$, $\tilde{c}$ and $\ve$ do not depend on $\delta$.
 Then, choosing $\delta$ small enough, we have \eqref{eq: lemconv Wqc}.
  Using the growth condition (W) we obtain $Dz\in L^{p}(B)$, so $z\in W^{1,p}(B)$.

Recall  that $a_0$ was chosen so that $\det Dz\in L^{1}(B)$ and $\cof Dz\in L^{q}(B)$.  Then Lemma  \ref{lem inner comp by lip} gives $z\in \A_{p,q}(B)$ and the proof is completed.
\end{proof}
In the following lemma we apply Lemma \ref{lem approx u by z} in the Lebesgue points of $Du$ and $\vn\circ u$. The proof is based on that of \cite[Lemma 3.3]{CoDo15}.
\begin{lem}\label{lem conv uj}
Let $\Omega\subset\R^{n}$ open, Lipschitz and bounded, and assume \emph{\ref{appr cond det>0})} or \emph{\ref{appr cond det=1})} of Lemma \ref{lem approx u by z}. Then for any 
 \[u\in\begin{cases}
 \A_{p,q}(B)&\text{in case } \emph{\ref{appr cond det>0})},\\
  \A_{p}^{1}(B)&\text{in case } \emph{\ref{appr cond det=1})}
 \end{cases}\]
 and any $\vn \in W^{1,s}(\imT(u,\Omega),\mS^{n-1})$,
  there are two sequences
\[u_j\in\begin{cases}
 \A_{p,q}(\Omega)&\text{in case \emph{\ref{appr cond det>0})},} \\
  \A_{p}^{1}(\Omega)&\text{in case \emph{\ref{appr cond det=1})}}
 \end{cases} \qquad \text{and} \quad \vn_j \in W^{1,s}(\imT(u,\Omega),\mS^{n-1}) \]  
such that $u_j\rightharpoonup u$ in $W^{1,p} (\Om, \Rn)$, $u_j=u$ on $\partial \Omega$, $\imT(u_j,\Omega) = \imT(u,\Omega)$ for all $j \in \N$, $\vn_j \weakc \vn$ in $W^{1,s} (\imT(u,\Omega), \mS^{n-1})$,
\[\limsup_{j\to\infty}\int_{\imT(u_j,\Omega)}V(\vn_j(y), D\vn_j(y)) \, \dd y\le \int_{\imT(u,\Omega)} V^{tqc}(\vn(y),D\vn(y)) \, \dd y \]
and 
\[\limsup_{j\to\infty}\int_{\Omega}W(Du_{j},\vn_j\circ u_j)\, \dd x\leq \int_{\Omega}W^{qc}(Du,\vn \circ u)\, \dd x.\]
If, additionally, $u\in W^{1,\infty}(\Omega,\R^{n})$, then we can take $u_j\in W^{1,\infty}(\Omega,\R^{n})$.
\end{lem}
\begin{proof}
We will only prove the case \emph{\ref{appr cond det>0})}, the proof of case \emph{\ref{appr cond det=1})} being completely analogous.
Thanks to Theorem \ref{th:tqc}, there exists a sequence $\{\vn_k\}_{k \in \N}$ in $W^{1,s}(\imT(u,\Omega),\mS^{n-1})$ such that $\vn_k\weakly\vn$ in $W^{1,s}(\imT(u,\Omega),\mS^{n-1})$ and 
\[\lim_{k\to\infty} \int_{\imT(u,\Omega)}V(\vn_k(y), D\vn_k(y)) \, \dd y =  \int_{\imT(u,\Omega)} V^{tqc}(\vn(y),D\vn(y)) \, \dd y. \]
 Fix $\eta\in (0,1)$.
 For the sequence $\{u_j\}_{j \in \N}$ it is enough to construct 
\[w\in\begin{cases}
 \A_{p,q}(\Omega)&\text{in case \emph{\ref{appr cond det>0})},} \\
  \A_{p}^{1}(\Omega)&\text{in case \emph{\ref{appr cond det=1})}} 
 \end{cases}\]
such that $\|u-w\|_{L^{p}}\leq \eta$, $w=u$ on $\partial\Omega$, $\imT(w,\Omega)=\imT(u,\Omega)$ and 
\begin{equation}\label{eq:convujiwlewqc}
\int_{\Omega}W(Dw,\vn\circ w)\, \dd x\leq \int_{\Omega}W^{qc}(Du,\vn\circ u)\, \dd x+\eta.
\end{equation}
Indeed for each $j \in \N$, we can construct $u_j$ as the $w$ of the claim above corresponding to $\eta=1/j$.
Then $u_j\to u$ in $L^{p}$ and, thanks to \eqref{eq:convujiwlewqc} and (W), we will have $\sup_{j \in \N} \|u_j\|_{W^{1,p}} < \infty$, so $u_j\weakly u$ in $W^{1,p} (\Om, \Rn)$. 
On the other hand, since $\vn_k\to \vn$ a.e.\ in $\imT(u,\Omega)$ and $u_j$ satisfies Luzin's $N^{-1}$ condition (i.e., the preimage of a set of measure zero has measure zero: this is a consequence of the fact that $\det D u_j > 0$ a.e.), for every $j\in\N$ we have $\vn_k\circ u_j\to \vn\circ u_j$ a.e.\ in $\Omega$ as $k\to\infty$, hence using (W) we obtain
\[\int_{\Omega}W(Du_j,\vn_k\circ u_j)\, \dd x\le \int_{\Omega} \left( h(|\vn_k\circ u_j-\vn\circ u_j|)+1 \right) W(Du_j,\vn\circ u_j)\, \dd x.\]
By dominated convergence, we have 
\[
 \limsup_{k \to \infty} \int_{\Omega} W(Du_j,\vn_k\circ u_j)\, \dd x\le \int_{\Omega} W(Du_j,\vn\circ u_j)\, \dd x ,
\]
so, for each $j\in\N$ we can take $k_j \in \N$ big enough to have 
\[\int_{\Omega}W(Du_j,\vn_{k_{j}}\circ u_j)\, \dd x\leq \int_{\Omega}W^{qc}(Du,\vn\circ u)\, \dd x+2j^{-1}.\]
Therefore, relabelling the sequence $\{\vn_j\}_{j\in\N}$ we have
\[\limsup_{j\to\infty}\int_{\Omega}W(Du_{j},\vn_j\circ u_j)\, \dd x\leq \int_{\Omega}W^{qc}(Du,\vn \circ u)\, \dd x.\]

If $\int_{\Omega}W^{qc}(Du,\vn\circ u)\, \dd x=\infty$, we can take $w=u$, so we will assume $W^{qc}(Du,\vn\circ u)\in L^{1}(\Omega)$. Using (W) we have
\[\frac{1}{c}|F|^{p}+\frac{1}{c}\theta(\det F)-c\le W^{qc}(F,\vm)\text{ for all }F\in \R^{n\times n}_+ \text{ and }\vm\in \mS^{n-1} .\]
 This is because the left-hand side of the inequality above is polyconvex, hence quasiconvex.
 Hence, $|Du|^{p}$ and $\theta(\det Du)$ are integrable.
  On the other hand, we have $\vn\circ u\in L^{\infty}(\Omega, \mS^{n-1})$, because thanks to \cite[Lemma 7.7]{BaHeMo17}, $\vn \circ u$ is measurable.
   Denote by $E$ the intersection of the set of  $p$-Lebesgue points of $Du$ and $\vn\circ u$ and Lebesgue points of $\theta(\det Du)$.
    Given $x\in E$, let $F_x=Du(x)$ and $\vm_x=\vn\circ u(x)$,  and choose $\delta_x$ as in Lemma \ref{lem approx u by z} for this $F_x$, $\vm_x$ and $\eta$ as above.

We will construct a sequence of $\{(w_j,\Omega_j)\}_{j\in\N}$ such that $w_j\in \A_{p,q}(\Omega)$, $\{\Omega_j\}_{j\in\N}$ is a decreasing sequence of open subsets of $\Omega$, $w_j=u$ on $\Omega_j$ and $\imT(w_j,\Omega)=\imT(u,\Omega)$.
 Set $w_0=u$ and $\Omega_0=\Omega$.
 The passage from $(w_j,\Omega_j)$ to $(w_{j+1},\Omega_{j+1})$ is as follows.
   For all $x\in E\cap\Omega_j$ we choose $r_j(x)\in (0,\eta)$ such that $B(x, r_j(x))\subset\Omega_j$, $u^* \in W^{1,p} (\partial B(x, r_j(x)), \Rn)$ (recall from Subsection \ref{subse:definitions} the definition of precise representative) and 
\[\fint_{B(x,r)}\left(|Dw_j(x')-F_x|^{p}+|\theta(\det Dw_j(x'))-\theta(\det F_x)|+|\vn\circ w_j(x')-\vm_x|^{p}\right) \dd x'\leq \delta_x\]
for all $r<r_j(x)$. The union of this collection of balls $B(x,r_j(x))$ covers $\Omega_j$ up to a set of measure zero.  Extract a finite disjoint subset $\{B(x_k,r_k)\}_{k=0}^{M}$ such that 
\[\left|\bigcup_{k=0}^{M}B(x_k,r_k) \right|\ge \frac{1}{2}|\Omega_j|.\]
Define $w_{j+1}$ as $w_j$ on $\Omega\setminus \bigcup_{k=0}^{M}B(x_k,r_k)$ and as the function $z$ of Lemma \ref{lem approx u by z} in each of the balls $B(x_k,r_k)$.
 Then $w_{j+1}=w_j=u$ on $\partial\Omega $ and  thanks to Lemma \ref{lem paste Apq}, we get 
\[w_{j+1}\in\begin{cases}
 \A_{p,q}(\Omega)&\text{if }W \text{ satisfies \emph{\ref{appr cond det>0})},} \\
  \A_{p}^{1}(\Omega)&\text{if }W \text{ satisfies \emph{\ref{appr cond det=1})}.}
 \end{cases}\]
Let $B (x_k',\frac{r_k}{2} )\subset B(x_k,r_k)$ be the ball given by Lemma \ref{lem approx u by z}.
Take an increasing sequence $\{U_i\}_{i\in\N}$ of open subsets compactly contained in $\Om$ such that $\bigcup_{i\in\N} U_{i}=\Omega$, $\bigcup_{k=0}^{M}B(x_k,r_k) \subset U_1$ and $w^*_j \in W^{1,p} (\partial U_i, \Rn)$ for all $i \in \N$.
Then, $w_j$ and $w_{j+1}$ coincide in a neighbourhood of each $\partial U_i$, so $w^*_{j+1} \in W^{1,p} (\partial U_i, \Rn)$ and $\imT(w_j,U_i) = \imT(w_{j+1},U_i)$ (recall Definition \ref{de:im top}), since the degree only depends on the boundary values.
Therefore, $\imT(w_j,\Omega) = \imT(w_{j+1},\Omega)$, and, by induction, $\imT(w_{j+1},\Omega)=\imT(u,\Omega)$.
 
By Lemma \ref{lem approx u by z},
 \begin{equation}\label{eq:sucesuj WleWqc}
 \int_{B\left(x_k',\frac{r_k}{2}\right)}W(Dw_{j+1},\vn\circ w_{j+1})\, \dd x\le \int_{B\left(x_k',\frac{r_k}{2}\right)}\left(W^{qc}(Du,\vn\circ u)+\eta\right) \dd x
 \end{equation}
 and 
  \begin{equation}\label{eq:sucesuj LpleWqc}
 \int_{B\left(x_k,r_k\right)}|w_{j+1}-u|^{p}\, \dd x\le c \, \eta^{p}\int_{B(x_k,r_k)}\left(W^{qc}(Du,\vn\circ u)+1\right) \dd x.
 \end{equation}
Set $\Omega_{j+1}=\Omega_{j}\setminus \bigcup_{k=0}^{M}\bar{B}\left(x_k',\frac{r_k}{2}\right)$.
 It is clear that $w_{j+1}=w_j=u$ on $\Omega_{j+1}$ and that $|\Omega_{j+1}|\le (1-2^{-n-1})|\Omega_j|$.
The construction of $w_{j+1}$ is completed and, hence, so is the sequence $\{ w_j \}_{j \in \N}$.
Thus, we only have to show that for $j$ big enough, $w_j$ has the desired properties, namely, \eqref{eq:convujiwlewqc} and that $\|u-w_j\|_{L^{p}}$ is small.

  Thanks to \eqref{eq:sucesuj LpleWqc} we have 
\[\int_{\Omega}|w_{j}-u|^{p}\, \dd x\le c \, \eta^{p} \int_{\Omega}\left(W^{qc}(Du,\vn\circ u)+1\right) \dd x , \]
so $w_j$ is close to $u$ in $L^{p}$, independently of $j$. 
On the other hand, from \eqref{eq:sucesuj WleWqc} we obtain
\[ \int_{\Omega\setminus\Omega_j}W(Dw_{j+1},\vn\circ w_{j+1})\, \dd x\le \int_{\Omega\setminus\Omega_j}\left(W^{qc}(Du,\vn\circ u)+\eta\right) \dd x,\]
which implies
 \[ \int_{\Omega}W(Dw_{j+1},\vn\circ w_{j+1})\, \dd x\le \int_{\Omega\setminus \Omega_j}\left(W^{qc}(Du,\vn\circ u)+\eta\right) \dd x+\int_{\Omega_j}W(Du,\vn\circ u)\, \dd x.\]
 Using $|\Omega_j|\le (1-2^{-n-1})^{j}|\Omega|\to 0$ and that, thanks to (W), we have $W(Du,\vn\circ u)\in L^{1}(\Omega)$ (since $|Du|^{p}$ and $\theta(\det Du)$ are integrable), for $j$ large enough we get 
 \begin{equation*}
 \int_{\Omega}W(Dw_{j+1},\vn\circ w_{j+1})\, \dd x\le \int_{\Omega}\left(W^{qc}(Du,\vn\circ u)+2\eta\right) \dd x
 \end{equation*}
 and the proof is concluded.
\end{proof}

\section{Relaxation}\label{se:generalizations}

Once the recovery sequence has been constructed in Theorem \ref{th conv nem det>0} and the lower semicontinuity and compactness results have been established in Section \ref{se:existence}, the general theory of relaxation (see, e.g., \cite[Th.\ 11.1.1 and 11.1.2]{AtBuMi06}) provides the following result.
We recall that the lower semicontinuous envelope is the largest lower semicontinuous function below a given one.

\begin{theorem}\label{th relax nem det>0}
Let $W$ satisfy (W) and let $V$ satisfy (V). Assume $W^{qc}$ is polyconvex.
Then $I^*$ is the lower semicontinuous envelope of $I$ with respect to the $L^1 (\Om, \R) \times L^1 (\Rn, \Rn)$ topology and, for each $(u,\vn) \in L^1 (\Om, \R) \times L^1 (\Rn, \Rn)$,
\[
 \displaystyle I^{*} (u,\vn)= \inf \left\{ \liminf_{j\to\infty} I(u_j,\vn_j) : (u_j, \vn_j) \to (u, \vn) \text{ as } j \to \infty \text{ in } L^1 (\Om, \R) \times L^1 (\Rn, \Rn) \right\} .
\]
If, in addition, $I$ is not identically infinity then 
\begin{enumerate}[a)]
\item There exists a minimizer of $I^*$.

\item Every minimizer of $I^*$ is the limit in $L^1 (\Om, \R) \times L^1 (\Rn, \Rn)$ of a minimizing sequence for $I$.

\item Every minimizing sequence of $I$ converges in $L^1 (\Om, \R) \times L^1 (\Rn, \Rn)$, up to a subsequence, to a minimizer of $I^*$.
\end{enumerate}
\end{theorem}

The analogue of Theorem \ref{th relax nem det>0} remains true in the incompressible case, i.e., when $W$ is assumed to satisfy (W$_1$) and every instance of $I$ is replaced by $I_1$, and every instance of $I^*$ by $I^*_1$.

As usual in relaxation and $\Gamma$-convergence problems (see, e.g., \cite[Rk.\ 2.2]{Braides06}), if $F$ is a functional continuous with respect to the topology $L^1 (\Om, \R) \times L^1 (\Rn, \Rn)$, then the relaxation of $I + F$ is $I^* + F$.
An example of such an $F$ is given by $F(u) = \int_{\Om} f(x, u(x)) \, \dd x$ with $f : \Om \times \Rn \to \R$ measurable in the first variable and continuous in the second such that $|f(x, y)| \leq C |y|^r + \gamma (x)$ for a.e.\ $x \in \Om$ and all $y \in \Rn$, for some $C>0$, $\gamma \in L^1 (\Om)$ and $0 \leq r < p^*$, where $p^*$ is the conjugate Sobolev exponent of $p$ (see, e.g., \cite[Cor.\ 6.51]{FoLe07}).
This is because, as shown in Proposition \ref{prop:lowersemicont} (in truth, \cite[Th.\ 8.2]{BaHeMo17}), if $u_j \to u$ in $L^1 (\Om, \Rn)$ and $\sup_{j \in \N} I_{\mec} (u_j, \vn_j) < \infty$ then, for a subsequence, $u_j \weakly u$ in $W^{1,p} (\Om, \Rn)$ and, by the compact Sobolev embedding, $u_j \to u$ in $L^r (\Om, \Rn)$.

\subsection*{Acknowledgements}

We thank G. Leoni for bringing to our notice the concept of tangential quasiconvexity. 
C.M.-C.\ has been supported by Project MTM2014-57769-C3-1-P and the ``Ram\'on y Cajal'' programme RYC-2010-06125 
of the Spanish Ministry of Economy and Competitivity.
Both authors have been supported by the ERC Starting grant no.\ 307179.

\small
\bibliographystyle{siam}
\bibliography{Bibliography}

\def\cprime{$'$}
\begin{thebibliography}{10}

\bibitem{AcFu84}
{\sc E.~Acerbi and N.~Fusco}, {\em Semicontinuity problems in the calculus of
  variations}, Arch. Rational Mech. Anal., 86 (1984), pp.~125--145.

\bibitem{AgDe11}
{\sc V.~Agostiniani and A.~DeSimone}, {\em {$\Gamma$}-convergence of energies
  for nematic elastomers in the small strain limit}, Contin. Mech. Thermodyn.,
  23 (2011), pp.~257--274.

\bibitem{AlLe01}
{\sc R.~Alicandro and C.~Leone}, {\em 3{D}-2{D} asymptotic analysis for
  micromagnetic thin films}, ESAIM Control Optim. Calc. Var., 6 (2001),
  pp.~489--498.

\bibitem{AtBuMi06}
{\sc H.~Attouch, G.~Buttazzo, and G.~Michaille}, {\em Variational analysis in
  {S}obolev and {BV} spaces}, SIAM and MPS, Philadelphia, PA, 2006.

\bibitem{Ball77}
{\sc J.~M. Ball}, {\em Convexity conditions and existence theorems in nonlinear
  elasticity}, Arch. Rational Mech. Anal., 63 (1977), pp.~337--403.

\bibitem{Ball81}
\leavevmode\vrule height 2pt depth -1.6pt width 23pt, {\em Global invertibility
  of {S}obolev functions and the interpenetration of matter}, Proc. Roy. Soc.
  Edinburgh Sect. A, 88 (1981), pp.~315--328.

\bibitem{BaCuOl81}
{\sc J.~M. Ball, J.~C. Currie, and P.~J. Olver}, {\em Null {L}agrangians, weak
  continuity, and variational problems of arbitrary order}, J. Funct. Anal., 41
  (1981), pp.~135--174.

\bibitem{BaJa87}
{\sc J.~M. Ball and R.~D. James}, {\em Fine phase mixtures as minimizers of
  energy}, Arch. Rational Mech. Anal., 100 (1987), pp.~13--52.

\bibitem{BaMu84}
{\sc J.~M. Ball and F.~Murat}, {\em {$W\sp{1,p}$}-quasiconvexity and
  variational problems for multiple integrals}, J. Funct. Anal., 58 (1984),
  pp.~225--253.

\bibitem{BaZa11}
{\sc J.~M. Ball and A.~Zarnescu}, {\em Orientability and energy minimization in
  liquid crystal models}, Arch. Rational Mech. Anal., 202 (2011), pp.~493--535.

\bibitem{BaDe15}
{\sc M.~Barchiesi and A.~DeSimone}, {\em Frank energy for nematic elastomers: a
  nonlinear model}, ESAIM Control Optim. Calc. Var., 21 (2015), pp.~372--377.

\bibitem{BaHeMo17}
{\sc M.~Barchiesi, D.~Henao, and C.~Mora-Corral}, {\em Local invertibility in
  {S}obolev spaces with applications to nematic elastomers and
  magnetoelasticity}, Arch. Rational Mech. Anal., 224 (2017), pp.~743--816.

\bibitem{Braides06}
{\sc A.~Braides}, {\em A handbook of ${\Gamma}$-convergence}, in Handbook of
  Differential Equations: Stationary Partial Differential Equations, M.~Chipot
  and P.~Quittner, eds., vol.~3, North-Holland, 2006, pp.~101--213.

\bibitem{CaGaYa15}
{\sc M.~C. Calderer, C.~A. Garavito~Garz\'on, and B.~Yan}, {\em {A Landau--de
  Gennes theory of liquid crystal elastomers}}, Discrete Contin. Dyn. Syst.
  Ser. S, 8 (2015), pp.~283--302.

\bibitem{CeDe11}
{\sc P.~Cesana and A.~DeSimone}, {\em Quasiconvex envelopes of energies for
  nematic elastomers in the small strain regime and applications}, J. Mech.
  Phys. Solids, 59 (2011), pp.~787--803.

\bibitem{CoDe03}
{\sc S.~Conti and C.~De~Lellis}, {\em Some remarks on the theory of elasticity
  for compressible {N}eohookean materials}, Ann. Sc. Norm. Super. Pisa Cl. Sci.
  (5), 2 (2003), pp.~521--549.

\bibitem{CoDo14}
{\sc S.~Conti and G.~Dolzmann}, {\em Relaxation of a model energy for the cubic
  to tetragonal phase transformation in two dimensions}, Math. Models Methods
  Appl. Sci., 24 (2014), pp.~2929--2942.

\bibitem{CoDo15}
\leavevmode\vrule height 2pt depth -1.6pt width 23pt, {\em On the theory of
  relaxation in nonlinear elasticity with constraints on the determinant},
  Arch. Rational Mech. Anal., 217 (2015), pp.~413--437.

\bibitem{Dacorogna82}
{\sc B.~Dacorogna}, {\em Quasiconvexity and relaxation of nonconvex problems in
  the calculus of variations}, J. Funct. Anal., 46 (1982), pp.~102--118.

\bibitem{Dacorogna08}
\leavevmode\vrule height 2pt depth -1.6pt width 23pt, {\em Direct methods in
  the calculus of variations}, vol.~78 of Applied Mathematical Sciences,
  Springer, New York, second~ed., 2008.

\bibitem{DaFoMaTr99}
{\sc B.~Dacorogna, I.~Fonseca, J.~Mal\'y, and K.~Trivisa}, {\em Manifold
  constrained variational problems}, Calc. Var. Partial Differential Equations,
  9 (1999), pp.~185--206.

\bibitem{DePr93}
{\sc P.~G. de~Gennes and J.~Prost}, {\em The Physics of Liquid Crystals},
  International Series of Monographs on Physics, Oxford University Press,
  Oxford, 2~ed., 1993.

\bibitem{Deimling85}
{\sc K.~Deimling}, {\em Nonlinear functional analysis}, Springer, Berlin, 1985.

\bibitem{DeDo02}
{\sc A.~DeSimone and G.~Dolzmann}, {\em Macroscopic response of nematic
  elastomers via relaxation of a class of {$\rm SO(3)$}-invariant energies},
  Arch. Rational Mech. Anal., 161 (2002), pp.~181--204.

\bibitem{DeTe09}
{\sc A.~DeSimone and L.~Teresi}, {\em Elastic energies for nematic elastomers},
  Eur. Phys. J. E, 29 (2009), pp.~191--204.

\bibitem{Fonseca88}
{\sc I.~Fonseca}, {\em The lower quasiconvex envelope of the stored energy
  function for an elastic crystal}, J. Math. Pures Appl. (9), 67 (1988),
  pp.~175--195.

\bibitem{FoGa95b}
{\sc I.~Fonseca and W.~Gangbo}, {\em Degree theory in analysis and
  applications}, Oxford University Press, New York, 1995.

\bibitem{FoLe07}
{\sc I.~Fonseca and G.~Leoni}, {\em Modern methods in the calculus of
  variations: {$L^p$} spaces}, Springer Monographs in Mathematics, Springer,
  New York, 2007.

\bibitem{GeLi59}
{\sc A.~N. Gent and P.~B. Lindley}, {\em Internal rupture of bonded rubber
  cylinders in tension}, Proc. Roy. Soc. London Ser. A, 249 (1959),
  pp.~195--205.

\bibitem{GiMoSo98I}
{\sc M.~Giaquinta, G.~Modica, and J.~Sou{\v{c}}ek}, {\em Cartesian currents in
  the calculus of variations. {I}}, Springer-Verlag, Berlin, 1998.

\bibitem{Hajlasz93}
{\sc P.~Haj{\l}asz}, {\em Change of variables formula under minimal
  assumptions}, Colloq. Math., 64 (1993), pp.~93--101.

\bibitem{Henao09}
{\sc D.~Henao}, {\em Cavitation, invertibility, and convergence of regularized
  minimizers in nonlinear elasticity}, J. Elasticity, 94 (2009), pp.~55--68.

\bibitem{HeMo10}
{\sc D.~Henao and C.~Mora-Corral}, {\em Invertibility and weak continuity of
  the determinant for the modelling of cavitation and fracture in nonlinear
  elasticity}, Arch. Rational Mech. Anal., 197 (2010), pp.~619--655.

\bibitem{HeMo11}
\leavevmode\vrule height 2pt depth -1.6pt width 23pt, {\em Fracture surfaces
  and the regularity of inverses for {BV} deformations}, Arch. Rational Mech.
  Anal., 201 (2011), pp.~575--629.

\bibitem{HeMo12}
\leavevmode\vrule height 2pt depth -1.6pt width 23pt, {\em Lusin's condition
  and the distributional determinant for deformations with finite energy}, Adv.
  Calc. Var., 5 (2012), pp.~355--409.

\bibitem{HeMo15}
\leavevmode\vrule height 2pt depth -1.6pt width 23pt, {\em Regularity of
  inverses of {S}obolev deformations with finite surface energy}, J. Funct.
  Anal., 268 (2015), pp.~2356--2378.

\bibitem{HeRo17}
{\sc D.~Henao and R.~Rodiac}, {\em On the existence of minimizers for the
  neo-{H}ookean energy in the axisymmetric setting}.
\newblock arXiv preprint 1609.07366.

\bibitem{KrStZe15}
{\sc M.~Kru\v{z}\'{\i}k, U.~Stefanelli, and J.~Zeman}, {\em Existence results
  for incompressible magnetoelasticity}, Discrete Contin. Dyn. Syst., 35
  (2015), pp.~2615--2623.

\bibitem{Morrey52}
{\sc C.~B. Morrey, Jr.}, {\em Quasi-convexity and the lower semicontinuity of
  multiple integrals}, Pacific J. Math., 2 (1952), pp.~25--53.

\bibitem{Morrey08}
\leavevmode\vrule height 2pt depth -1.6pt width 23pt, {\em Multiple integrals
  in the calculus of variations}, Classics in Mathematics, Springer-Verlag,
  Berlin, 2008.
\newblock Reprint of the 1966 edition.

\bibitem{Mucci09}
{\sc D.~Mucci}, {\em Relaxation of isotropic functionals with linear growth
  defined on manifold constrained {S}obolev mappings}, ESAIM Control Optim.
  Calc. Var., 15 (2009), pp.~295--321.

\bibitem{Muller88}
{\sc S.~M\"uller}, {\em Weak continuity of determinants and nonlinear
  elasticity}, C. R. Acad. Sci. Paris S\'er. I Math., 307 (1988), pp.~501--506.

\bibitem{Muller99}
\leavevmode\vrule height 2pt depth -1.6pt width 23pt, {\em Variational models
  for microstructure and phase transitions}, in Calculus of variations and
  geometric evolution problems ({C}etraro, 1996), vol.~1713 of Lecture Notes in
  Math., Springer, Berlin, 1999, pp.~85--210.

\bibitem{MuQiYa94}
{\sc S.~M\"uller, T.~Qi, and B.~S. Yan}, {\em On a new class of elastic
  deformations not allowing for cavitation}, Ann. Inst. H. Poincar\'e Anal. Non
  Lin\'eaire, 11 (1994), pp.~217--243.

\bibitem{MuSp95}
{\sc S.~M\"uller and S.~J. Spector}, {\em An existence theory for nonlinear
  elasticity that allows for cavitation}, Arch. Rational Mech. Anal., 131
  (1995), pp.~1--66.

\bibitem{RyLu05}
{\sc P.~Rybka and M.~Luskin}, {\em Existence of energy minimizers for
  magnetostrictive materials}, SIAM J. Math. Anal., 36 (2005), pp.~2004--2019.

\bibitem{SiSp00}
{\sc J.~Sivaloganathan and S.~J. Spector}, {\em On the existence of minimizers
  with prescribed singular points in nonlinear elasticity}, J. Elasticity, 59
  (2000), pp.~83--113.

\bibitem{SiSpTi06}
{\sc J.~Sivaloganathan, S.~J. Spector, and V.~Tilakraj}, {\em The convergence
  of regularized minimizers for cavitation problems in nonlinear elasticity},
  SIAM J. Appl. Math., 66 (2006), pp.~736--757.

\bibitem{Virga94}
{\sc E.~G. Virga}, {\em Variational theories for liquid crystals}, Applied
  Mathematics and Mathematical Computation, Chapman \& Hall, London, 1994.

\bibitem{VoGo76}
{\sc S.~K. Vodop'yanov and V.~M. Gol'd{\v{s}}te{\u\i}n}, {\em Quasiconformal
  mappings, and spaces of functions with first generalized derivatives},
  Sibirsk. Mat. \v Z., 17 (1976), pp.~515--531, 715.

\bibitem{Silhavy07}
{\sc M.~\v{S}ilhav\'y}, {\em Ideally soft nematic elastomers}, Netw. Heterog.
  Media, 2 (2007), pp.~279--311.

\bibitem{Sverak88}
{\sc V.~\v{S}ver\'ak}, {\em Regularity properties of deformations with finite
  energy}, Arch. Rational Mech. Anal., 100 (1988), pp.~105--127.

\bibitem{WaTe07}
{\sc M.~Warner and E.~Terentjev}, {\em Liquid Crystal Elastomers}, Clarendon
  Press, Oxford, 2007.

\bibitem{Ziemer89}
{\sc W.~P. Ziemer}, {\em Weakly differentiable functions}, vol.~120 of Graduate
  Texts in Mathematics, Springer-Verlag, New York, 1989.

\end{thebibliography}
\end{document}